\documentclass[12pt]{elsarticle}
\usepackage[english]{babel}
\usepackage{amsfonts,amsxtra,amsthm}

 \newcommand\dom{\operatorname{dom}}

\DeclareMathOperator{\re}{Re}
\DeclareMathOperator{\odd}{odd}
\DeclareMathOperator{\sign}{sign}
\DeclareMathOperator{\im}{Im}

\DeclareMathOperator{\eq}{eq}
     
\DeclareMathOperator{\li}{line}   
\DeclareMathOperator{\sech}{sech}
\DeclareMathOperator{\rad}{r}
\DeclareMathOperator{\half}{half}

\newcommand{\bb}[1]{\mathbf{#1}}
\newcommand{\mc}[1]{\mathcal{#1}}

\newcommand{\la}{\lambda}
\newcommand{\EE}{\mathcal E}  
\newtheorem{theorem}{Theorem}[section]
 
 \newtheorem{lemma}[theorem]{Lemma}
 \newtheorem{proposition}[theorem]{Proposition}
 \theoremstyle{definition}
 \newtheorem{definition}[theorem]{Definition}
 \theoremstyle{remark}
 \newtheorem{remark}[theorem]{Remark}
 
 \numberwithin{equation}{section}

\begin{document}

\title{Blow-up and strong instability of standing waves for the NLS-$\delta$ equation on a star graph}
%{On the blow up  of the standing wave solutions to NLS-$\delta$ equation on a star graph}

\author[ng]{Nataliia Goloshchapova\corref{cor}}
\ead{nataliia@ime.usp.br}
\author[mo]{Masahito Ohta}
\ead{mohta@rs.tus.ac.jp}
\cortext[cor]{Corresponding author}

\address[ng]{Department of Mathematics, IME-USP, Rua do Mat\~{a}o 1010, Cidade Universit\'{a}ria, CEP 05508-090, S\~{a}o Paulo, SP, Brazil}
\address[mo]{Department of Mathematics, Tokyo University of Science, 1-3 Kagurazaka, Shinjukuku, Tokyo 162-8601, Japan}

\begin{abstract}
We  study strong instability (by blow-up) of the standing waves for  the nonlinear Schr\"odinger equation with $\delta$-interaction on a star graph $\Gamma$. The key ingredient is a  novel variational technique applied to the standing wave solutions being minimizers of a specific variational problem. We also show well-posedness of the corresponding Cauchy problem in the domain of the self-adjoint operator which defines $\delta$-interaction. This permits to prove virial identity for the $H^1$- solutions to the Cauchy problem. We also prove certain strong instability results for the standing waves of the NLS-$\delta'$ equation on the line.
\end{abstract}

\maketitle

\section{Introduction}
Let $\Gamma$ be  a star graph, %$\Gamma$
i.e. $N$ half-lines $(0,\infty)$ joined at the vertex $\nu=0$. 
On $\Gamma$ we consider the following nonlinear Schr\"odinger equation 
with $\delta$-interaction (NLS-$\delta$)
\begin{equation}\label{NLS_graph_ger}
i\partial_t \mathbf{U}(t,x)-H\mathbf{U}(t,x) +|\mathbf{U}(t,x)|^{p-1}\mathbf{U}(t,x)=0,
\end{equation}
where $p>1$, $\mathbf{U}(t,x)=(u_j(t,x))_{j=1}^N:\mathbb{R}\times \mathbb{R}_+\rightarrow \mathbb{C}^N$,\,  nonlinearity acts componentwise,  i.e.  $(|\mathbf{U}|^{p-1}\mathbf{U})_j=|u_j|^{p-1}u_j$ and $H$ is the self-adjoint operator on $L^2(\Gamma)$ defined  by
\begin{equation}\label{D_alpha}
\begin{split}
(H \mathbf{V})(x)&=\left(-v_j''(x)\right)_{j=1}^N,\quad x> 0,\quad  \mathbf{V}=(v_j)_{j=1}^N,\\
\dom(H)&=\left\{\mathbf{V}\in H^2(\Gamma): v_1(0)=\ldots=v_N(0),\,\,\sum\limits_{j=1}^N  v_j'(0)=\alpha v_1(0)\right\}.
\end{split}
\end{equation}
\noindent Condition \eqref{D_alpha} is an analog of  $\delta$-interaction condition  for the 
Schr\"odinger operator on the line  (see \cite{AlbGes05}). 
On each edge of the graph (i.e. on each half-line) we have
$$i\partial_t u_j(t,x)+\partial_x^2u_j(t,x) +|u_j(t,x)|^{p-1}u_j(t,x)=0,\,\,x>0,\,\,j\in\{1,\ldots,N\},$$ 
moreover,  $\bb U(t,0)=(u_j(t,0))_{j=1}^N$  belongs to $\dom(H)$.

In the present paper we are aimed to study the strong instability of the standing wave solutions $\bb U(t,x)= e^{i\omega t}\bb\Phi(x)$ to \eqref{NLS_graph_ger}. It is easily seen that $\bb \Phi(x)$ satisfies  stationary equation \begin{equation}\label{H_alpha}
H\mathbf{\Phi}+\omega\mathbf{\Phi}-|\mathbf{\Phi}|
^{p-1}\mathbf{\Phi}=0.
\end{equation}
In \cite{AdaNoj14} the following description of the real-valued solutions to \eqref{H_alpha} was obtained.

\begin{theorem}
Let  $[s]$ denote the integer part of $s\in\mathbb{R}$, $\alpha\neq 0$.  Then equation \eqref{H_alpha} has $\left[\tfrac{N-1}{2}\right]+1$ (up to permutations of the edges of $\Gamma$) vector solutions $\mathbf{\Phi}_k^\alpha=(\varphi^\alpha_{k,j})_{j=1}^N, \,\,k=0,\ldots,\left[\tfrac{N-1}{2}\right]$, which are given by
\begin{equation}\label{Phi_k}
\begin{split}
 \varphi_{k,j}^\alpha(x)&= \left\{
                    \begin{array}{ll}
                      \Big[\frac{(p+1)\omega}{2} \sech^2\Big(\frac{(p-1)\sqrt{\omega}}{2}x-a_k\Big)\Big]^{\frac{1}{p-1}}, & \quad\hbox{$j=1,\ldots,k$;} \\
                     \Big[\frac{(p+1)\omega}{2} \sech^2\Big(\frac{(p-1)\sqrt{\omega}}{2}x+a_k\Big)\Big]^{\frac{1}{p-1}}, & \quad\hbox{$j=k+1,\ldots,N,$}
                    \end{array}
                  \right.\\
                  \text{where}\;\; a_k&=\tanh^{-1}\left(\frac{\alpha}{(2k-N)\sqrt{\omega}}\right),\,\,\text{and}\,\,\,\,\omega>\tfrac{\alpha^2}{(N-2k)^2}.
                  \end{split}
\end{equation}
\end{theorem}

\begin{definition} We say that $e^{i\omega t}\bb \Phi_k^\alpha$ is \textit{strongly unstable} if for any $\varepsilon > 0$ there 
exists $\bb U_0\in\mc E(\Gamma)$ such that $\|\bb U_0 -\bb\Phi_k^\alpha\|_{H^1(\Gamma)} <\varepsilon$ and the solution $\bb U(t)$ of \eqref{NLS_graph_ger} with $\bb U(0) = \bb U_0$ blows up in finite time (see definition of $\EE(\Gamma)$ in Notation section).
\end{definition}
Note that strong instability is the particular case of the\textit{ orbital instability} (which means that the corresponding solution to the Cauchy problem does not stay close to the orbit generated by the profile $\bb\Phi_k^\alpha$).
Study of the  orbital stability of the profiles $\mathbf{\Phi}_k^\alpha$ was initiated in \cite{AdaNoj14, AdaNoj16}. In particular, the authors considered the case $\alpha<\alpha^*<0,k=0$ (see Section \ref{sec3} for the interpretation of $\alpha^*$). They proved that  for $1<p\leq 5$ and $\omega>\frac{\alpha^2}{N^2}$ one gets orbital stability in $\EE(\Gamma)$, while for $p>5$ and $\omega>\omega^*>\frac{\alpha^2}{N^2}$ the standing wave is orbitally unstable.  The case of $k=0$ and $\alpha>0$ was considered in \cite{AngGol18, Kai17}. Essentially it had been proven that the standing wave is orbitally unstable for any $p>1$ and $\omega>\frac{\alpha^2}{N^2}$. The case of $\alpha\neq 0, k\neq 0$ was studied  in \cite{AngGol18a, Kai17}. 

The main  results of this paper are the following two strong instability theorems for $k=0$. 

\begin{theorem}\label{alpha>0}
Let  $\alpha>0$, $\omega>\frac{\alpha^2}{N^2}$,  and $p\geq 5$, then  
the standing wave $e^{i\omega t}\bb \Phi_0^\alpha(x)$ is  strongly unstable.
 \end{theorem}
  Observe that in   \cite[Theorem 1.1]{AngGol18}  the authors obtained orbital instability results only for $1<p<5$. Namely, the above theorem completes instability results for $p\geq 5$.   
\begin{theorem}\label{alpha<0}
Let $\alpha<0,\, p>5,\omega>\frac{\alpha^2}{N^2}$. 
Let $\xi_1(p)\in(0,1)$ be a unique solution of 
$$\frac{p-5}{2}\int\limits_\xi^1(1-s^2)^{\frac{2}{p-1}}ds=\xi(1-\xi^2)^{\frac{2}{p-1}},\quad (0<\xi<1),$$ and define $\omega_1=\omega_1(p,\alpha)=\frac{\alpha^2}{N^2\xi_1^2(p)}$. 
Then the standing wave solution $e^{i\omega t}\bb \Phi^\alpha_0$ is strongly unstable for all $\omega\in [\omega_1, \infty)$. 
\end{theorem}

To prove the above theorems  we use the  ideas by  \cite{FukayaO,Oht18}. It is worth mentioning that recently a lot of strong instability results have been obtained for different models based on the NLS equation (see \cite{Coz08, Oht18a, OhtYam15, OhtYam16} and references therein).  

 Classically the essential  ingredient in the proofs of blow-up results is the virial identity for the solution to the Cauchy problem with the initial data from the $L^2$-weighted space of the quadratic weight (see \cite[Chapter 6]{Caz03}). In Subsection \ref{D_H_vir} we prove the virial identity for the NLS-$\delta$ equation on $\Gamma$ using classical approach based on the approximation of $H^1$-initial data by the sequence of initial data functions with higher regularity. In particular, to do this we first prove the well-posedness of \eqref{NLS_graph_ger} in $\dom(H)$ with the norm $\|\cdot\|_H=\|(H+m)\cdot\|_{L^2}$ (here $H+m>0$). 
 
  Another important ingredient in the strong instability proofs is the variational characterization of the profile $\bb \Phi_0^\alpha$. In particular, this profile is the minimizer of the action functional $\bb S_\omega$ in the space $\mc E_{\eq}(\Gamma)$ 
restricted to the Nehari manifold. This characterization follows from the results obtained in \cite{FukJea08, FukOht08} for the NLS equation with $\delta$-interaction on the line. 

In Section 5 we apply our technique to show strong instability Theorems \ref{main_as} and \ref{main_odd} for the standing waves of the  NLS equation with attractive $\delta'$-interaction on the line. Their variational characterization has  been obtained in \cite{AdaNoj13}.

The paper is organized as follows.  In Section \ref{sec2} we prove well-posedness of the NLS-$\delta$ equation in $\dom(H)$ and show the virial identity as well.  Section \ref{sec3} is devoted to the variational characterization of the profile $\bb \Phi_0^\alpha$, while in Section \ref{sec4} we prove  Theorems \ref{alpha>0} and \ref{alpha<0}. In Section \ref{sec5} we consider the NLS-$\delta'$ equation on the line. In Section \ref{sec6} we show so-called \lq\lq product rule" for the derivative of the  unitary group $e^{iHt}$, which is strongly used in the proof of the well-posedness.    
\

\

\noindent{\bf Notation.}

The domain and the spectrum  of the operator $H$ are denoted by $\dom(H)$ and   $\sigma(H)$ respectively.

By $H^1_{\rad}(\mathbb{R})$ we denote the subspace of even functions in the Sobolev space $H^1(\mathbb{R})$. The dual space for $H^1(\mathbb{R}\setminus\{0\})$ is denoted by $H^{-1}(\mathbb{R}\setminus\{0\})$.

  On the star graph $\Gamma$ we define 
  \begin{equation*}
L^q(\Gamma)=\bigoplus\limits_{j=1}^N L^q(\mathbb{R}_+),\, q\ge 1, 
\quad H^1(\Gamma)=\bigoplus\limits_{j=1}^NH^1(\mathbb{R}_+),\quad H^2(\Gamma)=\bigoplus\limits_{j=1}^NH^2(\mathbb{R}_+). 
\end{equation*}  
For instance, the norm in $L^q(\Gamma)$ is  
$$\|\bb V\|^q_{L^q(\Gamma)}=\sum\limits_{j=1}^N\|v_j\|^q_{L^q(\mathbb{R}_+)},\quad \mathbf{V}=(v_j)_{j=1}^N.$$  
 By  $\|\cdot\|_q$ we  will  denote  the norm   in  $L^q(\cdot)$ (for the function on  $\Gamma$, or $\mathbb{R}$, or $\mathbb{R}_+$). 

We also define the spaces 
\begin{align*}
&\mathcal{E}(\Gamma)
=\{\bb V\in H^1(\Gamma): \,  v_1(0)=\ldots=v_N(0)\},\\
&\EE_{\eq}(\Gamma)
=\left\{ \bb V\in  \EE(\Gamma):\, v_1(x)=\ldots=v_N(x),\, x\in\mathbb{R}_+ \right\}.
\end{align*}
%\begin{equation*} \begin{split}
%L_k^2(\Gamma)=\left\{\begin{array}{c}\bb V\in L^2(\Gamma): v_1(x)=\ldots=v_k(x),\\
%v_{k+1}(x)=\ldots=v_N(x),\, x\in\mathbb{R}_+
%\end{array}\right\}, \quad \mathcal{E}(\Gamma)=\{\bb V\in H^1(\Gamma): \,  v_1(0)=\ldots=v_N(0)\},
%\end{split} \end{equation*}
Moreover, the dual space for $\EE(\Gamma)$ is denoted by $\EE'(\Gamma)$.
Finally, we set   $\Sigma(\Gamma)$ for  the following weighted Hilbert space 
\[
\Sigma(\Gamma)
%=\bigoplus\limits_{j=1}^NL^2(x^2, \mathbb{R}_+).
=\{ \bb V\in \EE(\Gamma): x \bb V\in L^2 (\Gamma)\}. 
\]
For   $\bb W=(w_j)_{j=1}^N$  on $\Gamma$, we will abbreviate
$$\int\limits_\Gamma\bb W dx=\sum\limits_{j=1}^N\int\limits_{\mathbb{R}_+}w_j dx.$$ 
 It is known that  $\inf \sigma(H)=\left\{\begin{array}{c}
0,\,\, \alpha\geq 0,\\
-\frac{\alpha^2}{N^2},\,\, \alpha<0.
\end{array}\right.$
Given the quantity 
\begin{equation*}\label{well_0}
0<m := 1 - 2\inf \sigma(H) <\infty,
\end{equation*} we introduce the norm $\|\bb\Psi\|_{H}:=\|(H + m)\bb\Psi\|_2$ that endows $\dom(H)$ with the structure of a Hilbert space. Observe that this norm for any real $\alpha$ is equivalent to $H^2$-norm on the graph.  Indeed, $$\|\bb\Psi\|^2_{H}=\|\bb \Psi''\|_2^2+m^2\|\bb\Psi\|^2_2+2m\|\bb\Psi'\|_2^2+2m\alpha|\psi_1(0)|^2.$$
Due to the choice of $m$ and the Sobolev embedding  we get $$C_1\|\bb \Psi\|_{H^2(\Gamma)}^2\leq \|\bb \Psi''\|_2^2+m\|\bb\Psi\|_2^2\leq \|\bb\Psi\|^2_{H}\leq C_2\|\bb \Psi\|_{H^2(\Gamma)}^2.$$ 
In what follows we will use the notation  $D_H=(\dom(H), \|\cdot\|_H)$.

By $C_j, C_j(\cdot), \,j\in\mathbb{N}$ and $C(\cdot)$ we will denote some positive constants. 

\section{Well-posedness}\label{sec2}
\subsection{Well-posedness in $H^1(\Gamma)$.}
It is known (see \cite{AdaNoj14, AngGol18, CacFin17}) that the Cauchy problem for equation \eqref{NLS_graph_ger} is well-posed. In particular, the following result holds.
\begin{theorem}\label{well_H1}
Let $p > 1$. Then for any $\bb U_0\in\mc E(\Gamma)$ there exists $T >
0$ such that equation \eqref{NLS_graph_ger} has a unique solution $\bb U(t)\in C([0, T], \mc E(\Gamma))\cap C^1([0, T],\mc E'(\Gamma))$ satisfying $\bb U(0) = \bb U_0$. For each $T_0\in (0, T)$ the mapping
$\bb U_0\in \mc E(\Gamma)\mapsto \bb U(t)\in C([0, T_0], \mc E(\Gamma))$ is continuous.  Moreover, the Cauchy problem for \eqref{NLS_graph_ger} has a maximal solution defined on
an interval of the form  $[0, T_{H^1})$, and the following “blow-up alternative” holds: either $T_{H^1} = \infty$ or $T_{H^1} < \infty$ and
$$\lim\limits_{t\to T_{H^1}}\|\bb U(t)\|_{H^1(\Gamma)} =\infty.$$
Furthermore, the solution $\bb U(t)$ satisfies 
\begin{equation}\label{conservation-laws}
\bb E(\bb U(t))=\bb E(\bb U_0),\quad \|\bb U(t)\|_2^2=\|\bb U_0\|_2^2
\end{equation}
for all $t\in[0, T_{H^1})$, where the energy is defined by
\begin{equation}\label{energy}\bb E(\bb V)
=\frac{1}{2}\|\bb V'\|_2^2+\frac{\alpha}{2}|v_1(0)|^2-\frac{1}{p+1}\|\bb V\|_{p+1}^{p+1}.
\end{equation}
\end{theorem}
\begin{remark}
Observe that for $1<p<5$ the global well-posedness holds due to the above conservation laws and  Gagliardo-Nirenberg inequality \eqref{G-N}.
\end{remark}

\subsection{Well-posedness in $D_H$ and virial identity}\label{D_H_vir}

\begin{theorem}\label{well_D_H}  Let $p\geq 4$ and $\bb U_0\in\dom(H)$. Then there exists $T>0$  such that equation  \eqref{NLS_graph_ger} has  a unique solution  $\bb U(t)\in C([0,T], D_H)\cap C^1([0,T], L^2(\Gamma))$ satisfying $\bb U(0)=\bb U_0$.  Moreover, the Cauchy problem for \eqref{NLS_graph_ger} has a maximal solution defined on
an interval of the form  $[0, T_H)$, and the following “blow-up alternative” holds: either $T_H = \infty$ or $T_H <\infty$ and
$$\lim\limits_{t\to T_H}\|\bb U(t)\|_H =\infty.$$
\end{theorem}
\begin{proof}
 Let $T>0$ to be chosen later. We will use the notation $$X_H=C([0,T], D_H)\cap C^1([0,T], L^2(\Gamma)),$$  and equip  the space $X_H$ with the norm
$$\|\bb U(t)\|_{X_H}=\sup\limits_{t\in[0,T]}\|\bb U(t)\|_H+\sup\limits_{t\in[0,T]}\|\partial_t \bb U(t)\|_2.$$
Consider $$E=\{\bb U(t)\in X_H:\, \bb U(0)=\bb U_0,\,\, \|\bb U(t)\|_{X_H}\leq M \},$$
where $M$ is a positive constant that will be chosen later as well. It is easily seen that  $(E, d)$ is a complete metric space with the  metric  $d(\bb U,\bb V)=\|\bb U-\bb V\|_{X_H}$. 
Now we consider the mapping defined by 
\begin{equation*}\label{duhamel}
\mc H(\bb U)(t)= \mc T(t)\bb U_0+ i\mc G(\bb U)(t),
\end{equation*}
where $\mc T(t)=e^{-iHt}$,\,\,$\mc G(\bb U)(t)=\int\limits_0^te^{-iH(t-s)}|\bb U(s)|^{p-1}\bb U(s)ds$, and  $\bb U\in E,\,\, t\in [0,T]$.

Our aim is to show that $\mc H$ is a contraction of $E$, and then to apply Banach's fixed point theorem. 

\textit{Step 1.}  We will show that $\mc H:\, E\to X_H$.

 {\bf 1.}\,  Recall that $\dom(H)=\{\bb \Psi\in L^2(\Gamma):\, \lim\limits_{h\to 0}h^{-1}(\mc T(h)-I)\bb \Psi\,\,\text{exists} \}.$ It is easily seen that $\bb W(t):=\mc T(t)\bb U_0\in \dom(H)$.  Hence $\partial_t \bb W(t)=-iHe^{-iHt}\bb U_0=-iH\bb W(t)$. Obviously $\partial_t\bb W(t)\in C([0,T], L^2(\Gamma))$ (due to the continuity of the unitary group $\mc T(t)$). The latter implies $$\|H(\bb W(t_n)-\bb W(t))\|_2\underset{t_n\to t}{\longrightarrow} 0,$$ where $t_n,t\in [0,T],$ and consequently $\bb W(t)\in X_H$.

 {\bf 2.}\, The inclusion  $\mc G(\bb U)(t)\in C^1([0,T],L^2(\Gamma))$ follows rapidly.
Indeed, \cite[Lemma 4.8.4]{Caz03} implies that $\partial_t(|\bb U(t)|^{p-1}\bb U(t))\in L^1([0,T], L^2(\Gamma))$, and the formula  
 \begin{equation}\label{well_2a}\partial_t\mc G(\bb U)(t)=ie^{-iHt}|\bb U(0)|^{p-1}\bb U(0)+i\int\limits_0^te^{-iH(t-s)}\partial_s(|\bb U(s)|^{p-1}\bb U(s))ds, 
 \end{equation} from the proof of \cite[Lemma 4.8.5]{Caz03} induces $\mc G(\bb U)(t)\in C^1([0,T],L^2(\Gamma)).$

 {\bf 3.} Below we will show that $\mc G(\bb U)(t)\in C([0,T], D_H)$.   First we need to prove that $\mc G(\bb U)(t)\in\dom(H).$  Note that 
 \begin{equation}\label{well_2c}
\|u|^{p-1}u-|v|^{p-1}v|\leq C(p)(|u|^{p-1}+|v|^{p-1})|u-v|,
\end{equation} 
which implies 
\begin{equation*}\label{well_8a}
\||\bb U|^{p-1}\bb U-|\bb V|^{p-1}\bb V\|_2\leq C_1(p)(\|\bb U\|_\infty^{p-1}+\|\bb V\|_\infty^{p-1})\|\bb U-\bb V\|_2.
\end{equation*}
 Therefore, by the  Gagliardo-Nirenberg inequality \begin{equation}\label{G-N}
\|\bb\Psi\|_q\leq C\|\bb\Psi'\|_2^{\frac{1}{2}-\frac{1}{q}}\|\bb\Psi\|_2^{\frac{1}{2}+\frac{1}{q}},\,\, q\in[2,\infty],\,\, \bb\Psi\in H^1(\Gamma),
\end{equation} for $\bb U, \bb V\in E$ we have
\begin{equation}\label{well_2}
\||\bb U|^{p-1}\bb U-|\bb V|^{p-1}\bb V\|_2\leq C(M)\|\bb U-\bb V\|_2,
\end{equation} 
where $C(M)$ is a positive constant depending on $M$. This implies $|\bb U(t)|^{p-1}\bb U(t)\in C([0,T],L^2(\Gamma))$.

 For $t\in [0, T)$ and $h\in(0, T-t]$ we get
 \begin{equation}\label{well_1}
 \begin{split} &\frac{\mc T(h)-I}{h}\mc G(\bb U)(t)=\frac{1}{h}\int_0^t\mc T(t+h-s)|\bb U(s)|^{p-1}\bb U(s)ds-\frac{1}{h}\int_0^t\mc T(t-s)|\bb U(s)|^{p-1}\bb U(s)ds\\&=\frac{\mc G(\bb U)(t+h)-\mc G(\bb U)(t)}{h}-\frac{1}{h}\int_t^{t+h}\mc T(t+h-s)|\bb U(s)|^{p-1}\bb U(s)ds.
 \end{split}
 \end{equation}
 Letting $h\to 0$, by the Mean Value Theorem, we arrive at  $H\mc G(\bb U)(t)=\mc G(\bb U)'(t)-|\bb U(t)|^{p-1}\bb U(t)$, i.e  we obtain the existence of the limit in \eqref{well_1},   and therefore $\mc G(\bb U)(t)\in \dom(H)$.    This is still true for $t=T$ since operator $H$ is closed.
 Note that we have used differentiability of $\mc G(\bb U)(t)$ proved above.

It remains to prove the continuity of $\mc G(\bb U)(t)$ in $H$-norm.
  We will use the integration by parts formula (it follows from Proposition \ref{leibniz})
 \begin{equation}\label{by_parts}
 \begin{split}
 &\mc G(\bb U)(t)=\int\limits_0^te^{-iH(t-s)}|\bb U(s)|^{p-1}\bb U(s)ds\\&=-i(H+m)^{-1}|\bb U(t)|^{p-1}\bb U(t)+ie^{-iHt}(H+m)^{-1}|\bb U(0)|^{p-1}\bb U(0)\\&+m(H+m)^{-1}\int\limits_0^te^{-iH(t-s)}|\bb U(s)|^{p-1}\bb U(s)ds+\frac{(p+1)i}{2}(H+m)^{-1}\int\limits_0^te^{-iH(t-s)}|\bb U(s)|^{p-1}\partial_s\bb U(s)ds\\&+\frac{(p-1)i}{2}(H+m)^{-1}\int\limits_0^te^{-iH(t-s)}\bb U^2(s)|\bb U(s)|^{p-3}\overline{\partial_s\bb U(s)}ds.
 \end{split}
\end{equation} 
Above we have used the formula
\begin{equation}\label{well_2b}
\begin{split}&\partial_t(|\bb U(t)|^{p-1}\bb U(t))=|\bb U(t)|^{p-1}\partial_t\bb U(t)+(p-1)\bb U(t)|\bb U(t)|^{p-3}\re(\overline{\bb U(t)}\partial_t\bb U(t))\\&=\tfrac{p+1}{2}|\bb U(t)|^{p-1}\partial_t\bb U(t)+\tfrac{p-1}{2}\bb U^2(t)|\bb U(t)|^{p-3}\overline{\partial_t\bb U(t)}.
\end{split}\end{equation}
  
Let $t_n, t\in [0,T]$, and $t_n\to t$. By \eqref{by_parts} we deduce
\begin{equation}\label{well_3}
\begin{split}
&\|\mc G(\bb U)(t)-\mc G(\bb U)(t_n)\|_H\leq \||\bb U(t)|^{p-1}\bb U(t)-|\bb U(t_n)|^{p-1}\bb U(t_n)\|_2\\&+m\int\limits_0^{t}\|\Big(e^{-iH(t-s)}-e^{-iH(t_n-s)}\Big)|\bb U(s)|^{p-1}\bb U(s)\|_2ds+m\left|\int\limits_t^{t_n}\|e^{-iH(t_n-s)}|\bb U(s)|^{p-1}\bb U(s)\|_2ds\right|\\&+\tfrac{p+1}{2}\int\limits_0^{t}\|\Big(e^{-iH(t-s)}-e^{-iH(t_n-s)}\Big)|\bb U(s)|^{p-1}\partial_s\bb U(s)\|_2ds+\tfrac{p+1}{2}\left|\int\limits_t^{t_n}\|e^{-iH(t_n-s)}|\bb U(s)|^{p-1}\partial_s\bb U(s)\|_2ds\right|\\&+\tfrac{p-1}{2}\int\limits_0^{t}\|\Big(e^{-iH(t-s)}-e^{-iH(t_n-s)}\Big)\bb U^2(s)|\bb U(s)|^{p-3}\overline{\partial_s\bb U(s)}\|_2ds\\&+\tfrac{p-1}{2}\left|\int\limits_t^{t_n}\|e^{-iH(t_n-s)}\bb U^2(s)|\bb U(s)|^{p-3}\overline{\partial_s\bb U(s)}\|_2ds\right|. 
\end{split}
\end{equation}
Therefore, using \eqref{well_2},\eqref{well_3}, unitarity and continuity properties of $e^{-iHt}$,
we obtain continuity of $\mc G(\bb U)(t)$ in $D_H$.

\noindent \textit{Step 2.} Now our aim is to choose  $T$ in order to guarantee invariance  of $E$ for the mapping $\mc H$, i.e.  $\mc H: E\to E$.

{\bf 1.} Using \eqref{well_2b}, we obtain  
\begin{equation}\label{well_4}
\begin{split}
&|\bb U(t)|^{p-1}\bb U(t)=\int\limits_0^t\partial_s\left(|\bb U(s)|^{p-1}\bb U(s)\right)ds+|\bb U(0)|^{p-1}\bb U(0)\\&=\int\limits_0^t\left\{\tfrac{p+1}{2}|\bb U(s)|^{p-1}\partial_s\bb U(s)+\tfrac{p-1}{2}\bb U^2(s)|\bb U(s)|^{p-3}\overline{\partial_s\bb U(s)}\right\}ds+|\bb U(0)|^{p-1}\bb U(0).
\end{split}
\end{equation}
Let $\bb U(t)\in E$ and $t\in [0,T]$. Using \eqref{G-N}, \eqref{by_parts}, \eqref{well_4},   and equivalence of $H$- and $H^2$-norms we obtain
\begin{equation}\label{well_5}
\begin{split}
&\|\mc H(\bb U)(t)\|_H\leq \|e^{-iHt}\bb U_0\|_H+\|\int\limits_0^te^{-iH(t-s)}|\bb U(s)|^{p-1}\bb U(s)ds\|_H\\&\leq  \|\bb U_0\|_H+ \|\int\limits_0^t\left\{\tfrac{p+1}{2}|\bb U(s)|^{p-1}\partial_s\bb U(s)+\tfrac{p-1}{2}\bb U^2(s)|\bb U(s)|^{p-3}\overline{\partial_s\bb U(s)}\right\}ds+|\bb U(0)|^{p-1}\bb U(0)\|_2\\&+\||\bb U(0)|^{p-1}\bb U(0)\|_2+m\int\limits_0^t\||\bb U(s)|^{p-1}\bb U(s)\|_2ds+\tfrac{p+1}{2}\int\limits_0^t\||\bb U(s)|^{p-1}\partial_s\bb U(s)\|_2ds\\&+\tfrac{p-1}{2}\int\limits_0^t\|\bb U^2(s)|\bb U(s)|^{p-3}\overline{\partial_s\bb U(s)}\|_2ds\\&\leq \|\bb U_0\|_H+C_1\|\bb U_0\|^p_H+C_2\int\limits_0^t\|\bb U\|_\infty^{p-1}\|\partial_s\bb U(s)\|_2ds+C_3\int\limits_0^t\|\bb U\|_\infty^{p-1}\|\bb U(s)\|_2ds\\&\leq \|\bb U_0\|_H+C_1\|\bb U_0\|^p_H+C_1(M)TM^p.
\end{split}
\end{equation}
{\bf 2.} Below  we will estimate $\|\partial_t\mc H(\bb U)(t)\|_2$.
Observe that 
\begin{equation}\label{well_6}
\|\partial_te^{-iHt}\bb U_0\|_2=\|H\bb U_0\|_2=\|\bb U''_0\|_2\leq \|\bb U_0\|_H.
\end{equation}
Using  \eqref{well_2a}, \eqref{G-N}, \eqref{well_2b}, \eqref{well_6},  we obtain the estimate
\begin{equation}\label{well_7}
\begin{split}
&\|\partial_t\mc H(\bb U)(t)\|_2\leq \|\bb U_0\|_H+\||\bb U(0)|^{p-1}\bb U(0)\|_2+\tfrac{p+1}{2}\int\limits_0^t\||\bb U(s)|^{p-1}\partial_s\bb U(s)\|_2ds\\&+\tfrac{p-1}{2}\int\limits_0^t\|\bb U^2(s)|\bb U(s)|^{p-3}\overline{\partial_s\bb U(s)}\|_2ds\leq \|\bb U_0\|_H+C_4\|\bb U_0\|^p_H+C_5\int\limits_0^t\|\bb U\|_\infty^{p-1}\|\partial_s\bb U(s)\|_2ds\\&\leq \|\bb U_0\|_H+C_4\|\bb U_0\|^p_H+C_2(M)TM^p.
\end{split}
\end{equation}
Finally, combining \eqref{well_5} and \eqref{well_7}, we arrive at
$$\|\mc H(\bb U)(t)\|_{X_H}\leq 2\|\bb U_0\|_H+(C_1+C_4)\|\bb U_0\|^p_H+(C_1(M)+C_2(M))TM^p.$$
We now let
$$\frac{M}{2}=\left(2\|\bb U_0\|_H+(C_1+C_4)\|\bb U_0\|^p_H\right).$$
By choosing  $T\leq\frac{1}{2(C_1(M)+C_2(M))M^{p-1}}$, we get 
$$\|\mc H(\bb U)(t)\|_{X_H}\leq M,$$ and therefore $\mc H: E\to E$.

\noindent\textit{Step 3.} Now we will choose $T$ to guarantee that $\mc H$ is a strict contraction on $(E,d)$. Let $\bb U, \bb V\in E$.

{\bf 1.}
\iffalse
Observe that for  $p\geq 3$ the nonlinearity $g(\bb U)=|\bb U|^{p-1}\bb U$ is of class $C^3$ and therefore $g''$ is Lipschitz on $H^2$-bounded sets. Therefore, using equivalence of $H$- and $H^2$-norms and following the ideas of the proof of \cite[Lemma 4.10.2]{Caz03} we get \fi
First, observe that  \eqref{well_2c} induces 
\begin{equation}\label{well_8}
\||\bb U|^{p-1}\bb U-|\bb V|^{p-1}\bb V\|_\infty\leq C_2(p)(\|\bb U\|_\infty^{p-1}+\|\bb V\|_\infty^{p-1})\|\bb U-\bb V\|_\infty.
\end{equation}
From \eqref{by_parts}, \eqref{well_4} it follows that 
\begin{equation}\label{well_9}
\begin{split}
&\|\mc H(\bb U)(t)-\mc H(\bb V)(t)\|_H=\|\int\limits_0^te^{-iH(t-s)}\left(|\bb U(s)|^{p-1}\bb U(s)-|\bb V(s)|^{p-1}\bb V(s)\right)ds\|_H\\&\leq
m\int\limits_0^t\||\bb U(s)|^{p-1}\bb U(s)-|\bb V(s)|^{p-1}\bb V(s)\|_2ds+(p+1)\int\limits_0^t\||\bb U(s)|^{p-1}\partial_s\bb U(s)-|\bb V(s)|^{p-1}\partial_s\bb V(s)\|_2ds\\&+(p-1)\int\limits_0^t\|\bb U^2(s)|\bb U(s)|^{p-3}\overline{\partial_s\bb U(s)}-\bb V^2(s)|\bb V(s)|^{p-3}\overline{\partial_s\bb V(s)}\|_2ds.
\end{split}
\end{equation}
To obtain the contraction property we need to estimate two last members of  inequality \eqref{well_9}.
Using convexity of the function $f(x)=x^\alpha, \, \alpha>1,\,x>0,$ one gets
\begin{equation*}\label{well_10}
|u|^{p-1}-|v|^{p-1}\leq (p-1)|u|^{p-2}|u-v|,\quad |u|\geq |v|,
\end{equation*}
and therefore 
\begin{equation}\label{well_11}
||u|^{p-1}-|v|^{p-1}|\leq (p-1)(|u|^{p-2}+|u|^{p-2})|u-v|.
\end{equation}
Using  \eqref{G-N} and \eqref{well_11}, we obtain
\begin{equation}\label{well_12}
\begin{split}
&\||\bb U(s)|^{p-1}\partial_s\bb U(s)-|\bb V(s)|^{p-1}\partial_s\bb V(s)\|_2\leq \||\bb U|^{p-1}(\partial_s\bb U-\partial_s \bb V)\|_2+\|\partial_s\bb V(|\bb U|^{p-1}-|\bb V|^{p-1})\|_2\\&\leq \|\bb U\|_\infty^{p-1}\|\partial_s\bb U-\partial_s \bb V\|_2+\|\partial_s\bb V\|_2\||\bb U|^{p-1}-|\bb V|^{p-1}\|_\infty\\&\leq
C_1M^{p-1}\|\bb U-\bb V\|_{X_H}+C_2M(\|\bb U\|^{p-2}_\infty+\|\bb V\|^{p-2}_\infty)\|\bb U-\bb V\|_\infty\leq C_3M^{p-1}\|\bb U-\bb V\|_{X_H}.
\end{split}
\end{equation}
Let us estimate the last term of \eqref{well_9}. Using  \eqref{well_8} and \eqref{well_11}, we get
\begin{equation}\label{well_13}
\begin{split}
&\|\bb U^2(s)|\bb U(s)|^{p-3}\overline{\partial_s\bb U(s)}-\bb V^2(s)|\bb V(s)|^{p-3}\overline{\partial_s\bb V(s)}\|_2\\&\leq \|\bb U^2|\bb U|^{p-3}(\overline{\partial_s\bb U}-\overline{\partial_s\bb V})\|_2+\|\overline{\partial_s\bb V}(\bb U^2|\bb U|^{p-3}-\bb V^2|\bb V|^{p-3})\|_2\\&\leq\|\bb U\|^{p-1}_\infty\|\partial_s\bb U-\partial_s\bb V\|_2+\|\left(|\bb U|^{p-3}\bb U-|\bb V|^{p-3}\bb V\right)(\bb U+\bb V)\overline{\partial_s\bb V}\|_2\\&+\|\bb U\bb V\left(|\bb U|^{p-3}-|\bb V|^{p-3}\right)\overline{\partial_s\bb V}\|_2\\&\leq C_4M^{p-1}\|\bb U-\bb V\|_{X_H}+\|\partial_s\bb V\|_2\|\bb U+\bb V\|_\infty\||\bb U|^{p-3}\bb U-|\bb V|^{p-3}\bb V\|_\infty\\&+C_5\|\bb U\|_\infty\|\bb V\|_\infty\left(\|\bb U\|^{p-4}_{\infty}
+\|\bb V\|_\infty^{p-4}\right)\|\bb U-\bb V\|_\infty\|\partial_s\bb V\|_2\leq C_6M^{p-1}\|\bb U-\bb V\|_{X_H}.
\end{split}
\end{equation}
Finally, combining \eqref{well_2c},\eqref{well_9},\eqref{well_12},\eqref{well_13}, we obtain
\begin{equation}\label{well_14}
\|\mc H(\bb U)-\mc H(\bb V)\|_H\leq C_7M^{p-1}T\|\bb U-\bb V\|_{X_H}.
\end{equation}

{\bf 2.} To get the contraction property of $\mc H$ we need to estimate $L^2$-part of $X_H$-norm of $\mc H(\bb U)(t)-\mc H(\bb V)(t)$.
From \eqref{well_2a}, we deduce
\begin{equation}\label{well_15}
\|\partial_t\mc H(\bb U)(t)-\partial_t\mc H(\bb V)(t)\|_2\leq \int\limits_0^t\|\partial_s(|\bb U(s)|^{p-1}\bb U(s))-\partial_s(|\bb V(s)|^{p-1}\bb V(s))\|_2ds.
\end{equation}
Using \eqref{well_2b},\eqref{well_12},\eqref{well_13}, from \eqref{well_15}
we get 
\begin{equation}\label{well_16}
\|\partial_t\mc H(\bb U)(t)-\partial_t\mc H(\bb V)(t)\|_2\leq C_8M^{p-1}T\|\bb U-\bb V\|_{X_H}
\end{equation} and finally from \eqref{well_14},\eqref{well_16}, we  obtain
$$\|\mc H(\bb U)(t)-\mc H(\bb V)(t)\|_{X_H}\leq (C_7+C_8)M^{p-1}T\|\bb U-\bb V\|_{X_H}.
$$
Thus, for $$T<\min\left\{\frac{1}{(C_7+C_8)M^{p-1}}, \frac{1}{2(C_1(M)+C_2(M))M^{p-1}}\right\}$$ 
the mapping $\mc H$ is the strict contraction of $(E,d)$. Therefore, by the  Banach fixed point theorem,  $\mc H$ has a unique fixed point $\bb U\in E$ which is a solution of \eqref{NLS_graph_ger}.

Uniqueness of the solution follows in a standard way. Suppose that $\bb U_1(t)$ and $\bb U_2(t)$ are two solutions to \eqref{NLS_graph_ger}, and $\widetilde M=\sup\limits_{t\in[0,T]}\max\{\|\bb U_1(t)\|_H, \|\bb U_2(t)\|_H\}$. Then 
\begin{align*}
&\|\bb U_1(t)-\bb U_2(t)\|_2=\|\int\limits_0^t e^{-iH(t-s)}\Bigl(|\bb U_1(s)|^{p-1}\bb U_1(s)-|\bb U_2(s)|^{p-1}\bb U_2(s)\Bigr)ds\|_2\\ &\leq C(\widetilde M)\int\limits_0^t\|\bb U_1(s)-\bb U_2(s)\|_2ds,
\end{align*}
and the result follows from Gronwall's lemma.
The blow-up alternative can be shown by bootstrap.
\end{proof}
\begin{remark}
\begin{itemize}
\item[$(i)$]
 The assumption $p\geq 4$ is technical. We believe that the result also holds for the smaller values of   $p$ (see \cite[Subsection 4.12]{Caz03}).
 \item[$(ii)$] The idea of the proof of the above theorem was given in \cite{CacFin17} (see Proposition 2.5) without details.
 \end{itemize}
\end{remark}
Below we will show the virial identity which is crucial for the proof of the strong instability.
Define 
\begin{equation}\label{well_16a}
\bb P(\bb V)=\|\bb V'\|_2^2+\frac{\alpha}{2}|v_1(0)|^2-\frac{p-1}{2(p+1)}\|\bb V\|_{p+1}^{p+1},\quad \bb V\in \mc E(\Gamma).
\end{equation}

\begin{proposition}\label{prop-virial} Assume that  $\Sigma(\Gamma)$ is the weighted space given in Notation section.
Let $\bb U_0\in \Sigma(\Gamma)$, 
and let $\bb U(t)$ be the corresponding maximal solution to \eqref{NLS_graph_ger}. 
Then $\bb U(t)\in C([0, T_{H^1}), \Sigma(\Gamma))$, moreover,
the function 
\begin{equation*}\label{well_17}
f(t):=\int\limits_\Gamma x^2|\bb U(t,x)|^2dx
\end{equation*}
belongs to $C^2[0,T_{H^1})$,
\begin{equation}\label{well_17a}
f'(t)=4\im \int\limits_\Gamma x\overline{\bb U}\partial_x\bb Udx,
\end{equation}
and 
\begin{equation}\label{well_18}
f''(t)=8\bb P(\bb U(t)) \quad\text{(virial identity)}
\end{equation}
for all $t\in [0, T_{H^1})$.
\end{proposition}

\begin{proof}
 The proof is similar to the one of \cite[Proposition 6.5.1]{Caz03}. 
 We give it for convenience of the reader.

\textit{Step 1.} \,Let $\varepsilon>0$, define $f_\varepsilon(t)=\|e^{-\varepsilon x^2}x\bb U(t)\|^2_2$, for $t\in [0,T], \,\, T\in (0, T_{H^1})$. Then, noting that $e^{-2\varepsilon x^2}x^2\bb U(t)\in H^1(\Gamma)$, we get
\begin{equation}\label{well_19}
\begin{split}
f'_\varepsilon(t)&=2\re\int\limits_\Gamma e^{-2\varepsilon x^2}x^2\overline{\bb U}\partial_t\bb Udx=2\re\int\limits_\Gamma  e^{-2\varepsilon x^2}x^2\overline{\bb U}\left(i\partial_x^2\bb U+i|\bb U|^{p-1}\bb U\right)dx\\&=
-2\im \int\limits_\Gamma  e^{-2\varepsilon x^2}x^2\overline{\bb U}\partial_x^2\bb U dx= 4\im \int\limits_\Gamma \left\{e^{-\varepsilon x^2}(1-2\varepsilon x^2)\right\}\overline{\bb U} xe^{-\varepsilon x^2}\partial_x\bb Udx.
\end{split}
\end{equation}
Observe that $|e^{-\varepsilon x^2}(1-2\varepsilon x^2)|\leq C(\varepsilon)$ for any $x$. From \eqref{well_19}, by the Cauchy-Schwarz inequality, we obtain
\begin{equation}\label{well_20}
\begin{split}
&|f'_\varepsilon(t)|\leq 4\left|\int\limits_\Gamma\left\{e^{-\varepsilon x^2}(1-2\varepsilon x^2)\right\}\overline{\bb U} xe^{-\varepsilon x^2}\partial_x\bb U dx\right|\leq 4C(\varepsilon)\int\limits_\Gamma|e^{-\varepsilon x^2}x\bb U\partial_x\bb U|dx \\&\leq 4C(\varepsilon)\sum\limits_{j=1}^N\|\partial_xu_j\|_2\|e^{-\varepsilon x^2}xu_j\|_2\leq C(\varepsilon,N)\|\bb U\|_{H^1(\Gamma)}\sqrt{f_\varepsilon(t)}.\end{split}
\end{equation}
From \eqref{well_20} one implies
$$\int\limits_0^t\frac{f'_\varepsilon(s)}{\sqrt{f_\varepsilon(s)}}ds\leq C(\varepsilon,N)\int\limits_0^t\|\bb U(s)\|_{H^1(\Gamma)}ds,$$
and therefore 
$$\sqrt{f_\varepsilon(t)}\leq \|x\bb U_0\|_2+\frac{C(\varepsilon,N)}{2}\int\limits_0^t\|\bb U(s)\|_{H^1(\Gamma)}ds,\,\,t\in [0,T].$$
 Letting $\varepsilon\downarrow 0$ and applying  Fatou's lemma, we get that $x\bb U(t)\in L^2(\Gamma)$ and $f(t)$  is bounded in $[0,T].$
 Observe that from \eqref{well_19} one induces
 \begin{equation}\label{well_21}
  f_\varepsilon(t)=f_\varepsilon(0)+4\im\int\limits_0^t\int\limits_\Gamma \left\{e^{-\varepsilon x^2}(1-2\varepsilon x^2)\right\}\overline{\bb U} xe^{-\varepsilon x^2}\partial_x\bb Udx ds. \end{equation}
  We have the following estimates for any positive $x$ and $\varepsilon$:
  \begin{equation}\label{well_22}
  \begin{split}
  &e^{-2\varepsilon x^2}x^2|\bb U(t)|^2\leq x^2|\bb U(t)|^2,\\
  &e^{-2\varepsilon x^2}x^2|\bb U_0|^2\leq x^2|\bb U_0|^2,\\
  & |e^{-\varepsilon x^2}(1-2\varepsilon x^2)\overline{\bb U} xe^{-\varepsilon x^2}\partial_x\bb U|\leq C(\varepsilon)|\partial_x\bb U\|x\bb U|.
  \end{split}
  \end{equation}
  Having pointwise convergence, and using  \eqref{well_22},  by the Dominated Convergence Theorem we get from \eqref{well_21}
  $$f(t)= \|x\bb U(t)\|_2^2= \|x\bb U_0\|_2^2+4 \im\int\limits_0^t\int\limits_\Gamma x\overline{\bb U}\partial_x\bb Udx ds.$$
  Since $\bb U(t)$ is strong $H^1$-solution, $f(t)$ is $C^1$-function, and \eqref{well_17a} holds for any $t\in [0, T_{H^1}).$  
  
  Using continuity of $\|x\bb U(t)\|_2$ and the inclusion $\bb U(t)\in C([0, T_{H^1}), \mc E(\Gamma))$, we get $\bb U(t)\in C([0, T_{H^1}), \Sigma(\Gamma)).$
  
  \textit{Step 2. } Let $\bb U_0\in \dom(H)$. By Theorem \ref{well_D_H}, the solution $\bb U(t)$ to the corresponding Cauchy problem belongs to $C([0,T_H), D_H)\cap C^1([0,T_H), L^2(\Gamma))$. Following the ideas of proofs of \cite[Theorem 5.3.1, Theorem 5.7.1]{Caz03} and using Strichartz estimate from \cite[Theorem 1.3]{BanIgn14}, one can show that $T_{H^1}=T_H$.  
  
  Let $\varepsilon>0$ and $\theta_\varepsilon(x)=e^{-\varepsilon x^2}$. Define 
    \begin{equation}\label{well_22a}
  h_\varepsilon(t)=\im \int\limits_\Gamma\theta_\varepsilon x\overline{\bb U}\partial_x\bb Udx\,\, \text{for}\,\,t\in[0,T],\,\, T\in (0, T_H).\end{equation}
  First, let us show that 
  \begin{equation}\label{well_23}
  h'_\varepsilon(t)=-\im\int\limits_\Gamma\partial_t\bb U\left\{2\theta_\varepsilon x\overline{\partial_x\bb U}+(\theta_\varepsilon+x \theta'_\varepsilon)\overline{\bb U}\right\}dx\end{equation} or equivalently
  \begin{equation}\label{well_24}
   h_\varepsilon(t)= h_\varepsilon(0)-\im\int\limits_0^t\int\limits_\Gamma\partial_s\bb U\left\{2\theta_\varepsilon x\overline{\partial_x\bb U}+(\theta_\varepsilon+x \theta'_\varepsilon)\overline{\bb U}\right\}dx ds.
  \end{equation} 
  Let us prove that identity \eqref{well_24} holds for $\bb U(t)\in C([0,T], H^1(\Gamma))\cap C^1([0,T], L^2(\Gamma)).$ Note that by density argument it is sufficient to show \eqref{well_24} for $\bb U(t)\in C^1([0,T], H^1(\Gamma))\cap C^1([0,T], L^2(\Gamma)).$
  From \eqref{well_22a}, it follows
  \begin{equation}\label{well_24a} h'_\varepsilon(t)=-\im\int\limits_\Gamma\left\{\theta_\varepsilon x\partial_t\bb U\overline{\partial_x\bb U}+\theta_\varepsilon x\bb U\overline{\partial_{xt}^2\bb U}\right\}dx. \end{equation}
  Note that 
  $$\theta_\varepsilon x\bb U\overline{\partial_{xt}^2\bb U}=\theta_\varepsilon x\bb U\overline{\partial_{tx}^2\bb U}=\partial_x\left(\theta_\varepsilon x\bb U\overline{\partial_{t}\bb U}\right)-\theta_\varepsilon \bb U\overline{\partial_{t}\bb U}-\theta_\varepsilon x\partial_x\bb U\overline{\partial_{t}\bb U}-x\theta'_\varepsilon\bb U\overline{\partial_{t}\bb U}, $$
  which induces
  $$
\int\limits_{\Gamma}  \theta_\varepsilon x\bb U\overline{\partial_{xt}^2\bb U}dx=-\int\limits_{\Gamma} \overline{\partial_{t}\bb U}\left\{\theta_\varepsilon(\bb U+x\partial_x\bb U)+x\theta'_\varepsilon\bb U\right\}dx.
$$
Therefore, from \eqref{well_24a}, we get 
$$  h'_\varepsilon(t)=-\im\int\limits_\Gamma\left\{\theta_\varepsilon x\partial_t\bb U\overline{\partial_x\bb U}+\partial_t\bb U\left(\theta_\varepsilon(\overline{\bb U}+x\overline{\partial_x\bb U})+x\theta'_\varepsilon\overline{\bb U}\right)\right\}dx.$$ 
Consequently we obtain \eqref{well_24} for  $\bb U(t)\in C^1([0,T], H^1(\Gamma))\cap C^1([0,T], L^2(\Gamma))$ and hence for $\bb U(t)\in C([0,T], H^1(\Gamma))\cap C^1([0,T], L^2(\Gamma))$ which implies \eqref{well_23}.

 Since $\bb U(t)\in C([0,T_{H}), D_H)$, from \eqref{well_23} we get
 \begin{equation}\label{well_25}
  h'_\varepsilon(t)=-\re\int\limits_\Gamma(-H\bb U+|\bb U|^{p-1}\bb U)\left\{2\theta_\varepsilon x\overline{\partial_x\bb U}+(x\theta_\varepsilon)'\overline{\bb U}\right\}dx.
 \end{equation}
 Below we will consider separately  linear and nonlinear part of identity \eqref{well_25}. Integrating by parts, we obtain
 \begin{equation}\label{well_26}
 \begin{split}
& -\re\int\limits_\Gamma-H\bb U\left\{2\theta_\varepsilon x\overline{\partial_x\bb U}+(x\theta_\varepsilon)'\overline{\bb U}\right\}dx\\&=\alpha|u_1(0)|^2+2\int\limits_\Gamma x\theta'_\varepsilon|\partial_x\bb U|^2dx+\int\limits_\Gamma(2\theta'_\varepsilon+x\theta''_\varepsilon)\re(\overline{\bb U}\partial_x\bb U)dx+2\int\limits_\Gamma \theta_\varepsilon|\partial_x\bb U|^2dx,
 \end{split}
 \end{equation}
  and
  \begin{equation}\label{well_27} 
  \begin{split}
  & -\re\int\limits_\Gamma|\bb U|^{p-1}\bb U\left\{2\theta_\varepsilon x\overline{\partial_x\bb U}+(x\theta_\varepsilon)'\overline{\bb U}\right\}dx\\&=-\int\limits_\Gamma|\bb U|^{p+1}\theta_\varepsilon dx-\int\limits_\Gamma|\bb U|^{p+1}x\theta'_\varepsilon dx-\int\limits_\Gamma(|\bb U|^2)^{\frac{p-1}{2}}\partial_x(|\bb U|^2)x\theta_\varepsilon  dx\\&=-\frac{p-1}{p+1}\int\limits_\Gamma|\bb U|^{p+1}\theta_\varepsilon dx-\frac{p-1}{p+1}\int\limits_\Gamma|\bb U|^{p+1}x\theta'_\varepsilon dx. 
  \end{split}
  \end{equation}
  Finally, from \eqref{well_25}-\eqref{well_27} we get 
  \begin{equation*}
  \begin{split}
 & h'_\varepsilon(t)=\left[2\int\limits_\Gamma\theta_\varepsilon|\partial_x\bb U|^2dx +\alpha|u_1(0)|^2-\frac{p-1}{p+1}\int\limits_\Gamma|\bb U|^{p+1}\theta_\varepsilon dx\right]\\&+\left[2\int\limits_\Gamma x\theta'_\varepsilon|\partial_x\bb U|^2dx+\int\limits_\Gamma(2\theta'_\varepsilon+x\theta''_\varepsilon)\re(\overline{\bb U}\partial_x\bb U)dx\right]-\frac{p-1}{p+1}\int\limits_\Gamma|\bb U|^{p+1}x\theta'_\varepsilon dx. 
  \end{split}
  \end{equation*}
  Since $\theta_\varepsilon,\, \theta'_\varepsilon,\, x\theta'_\varepsilon,\, x\theta''_\varepsilon$ are bounded with respect to $x$ and $\varepsilon$, and 
   $$ \theta_\varepsilon\to 1,\,\, \theta'_\varepsilon\to 0,\,\,x\theta'_\varepsilon\to 0,\,\,x\theta''_\varepsilon\to 0\,\, \text{pointwise as}\,\, \varepsilon\downarrow 0, $$
   by the Dominated Convergence Theorem we have
   $$\lim\limits_{\varepsilon\downarrow 0}h'_\varepsilon(t)=2\|\partial_x\bb U\|_2^2+\alpha|u_1(0)|^2-\frac{p-1}{p+1}\|\bb U\|_{p+1}^{p+1}=:g(t).$$
      Moreover,  again by the Dominated Convergence Theorem, 
      $$\lim\limits_{\varepsilon\downarrow 0}h_\varepsilon(t)=\im \int\limits_\Gamma x\overline{\bb U}\partial_x\bb Udx=:h(t).$$
      Using continuity of $g(t)$ and the fact that operator $A=\dfrac{d}{dt}$ in  the space $C[0,T]$ with $\dom(A)=C^1[0,T]$  is closed, we arrive at  $h'(t)=g(t),\,\, t\in[0,T]$, i.e.
$$ h'(t)= 2\|\partial_x\bb U\|_2^2+\alpha|u_1(0)|^2-\frac{p-1}{p+1}\|\bb U\|_{p+1}^{p+1},$$
and $h(t)$ is $C^1$ function. Finally, \eqref{well_18} holds for $\bb U_0\in \dom(H)$.

 To conclude the proof   consider $\{\bb U_0^n\}_{n\in \mathbb{N}}\subset \dom(H)$ such that $\bb  U_0^n\to \bb U_0$ in $H^1(\Gamma)$ and $x\bb U_0^n\to x\bb U_0$ in $L^2(\Gamma)$ as $n\to \infty$. Let $\bb U^n(t)$ be the maximal solutions of the corresponding Cauchy problem associated with \eqref{NLS_graph_ger}.
 From \eqref{well_17a} and \eqref{well_18} we obtain 
 $$\|x\bb U^n(t)\|_2^2=\|x\bb U^n_0\|_2^2+4t\im \int\limits_\Gamma x\overline{\bb U_0^n}\partial_x\bb U_0^ndx+\int_0^t\int_0^s8P(\bb U^n(y))dyds.$$
Using continuous dependence and repeating the arguments from \cite[Corollary 6.5.3]{Caz03}, we obtain as $n\to \infty$
$$\|x\bb U(t)\|_2^2=\|x\bb U_0\|_2^2+4t\im \int\limits_\Gamma x\overline{\bb U_0}\partial_x\bb U_0dx+\int_0^t\int_0^s8P(\bb U(y))dyds,$$ that is \eqref{well_18} holds for  $\bb U_0\in \mc E(\Gamma)$.
 \end{proof}
\begin{remark}
In \cite{CozFuk08} the authors proved the virial identity for the NLS equation with $\delta$-potential on the line using approximation of $\delta$-potential by smooth potentials $V_{\varepsilon}(x) =\tfrac1{\varepsilon} e^{-\pi\tfrac1{\varepsilon^2} x^2},\, \varepsilon\to 0$, and applying the virial identity to the NLS equation on $\mathbb{R}$ with the smooth potential (which is classical). Observe that in the present paper we overcome this procedure by proving the well-posedness in $D_H$. Obviously our proof can be repeated for  the NLS equation with $\delta$-potential on the line.
\end{remark}

\section{Variational analysis}\label{sec3}

Define the following action functional
\begin{equation}\label{well_27b}
\bb S_\omega(\bb V)
=\frac{1}{2}\| \bb V' \|_2^2
+\frac{\omega}{2} \| \bb V \|_2^2
-\frac{1}{p+1} \|\bb V \|_{p+1}^{p+1}
+\frac{\alpha}{2}|v_1(0)|^2.
\end{equation}

We also introduce
\begin{equation*}
\bb I_\omega(\bb V)
=\| \bb V' \|_2^2+\omega \|\bb V \|_2^2-\| \bb V \|_{p+1}^{p+1}
+\alpha |v_1(0)|^2.
\end{equation*}

Observe that
$$\bb I_\omega(\bb V)=\partial_{\la}\bb S_\omega(\la\bb V)|_{\la=1}
=\langle \bb S'_\omega(\bb V),\bb V\rangle,$$ and
\begin{equation}\label{var_1}
\bb S_\omega(\bb V)=\frac{1}{2}\bb I_\omega(\bb V)+\frac{p-1}{2(p+1)}\|\bb V\|_{p+1}^{p+1}.
\end{equation}
  In \cite{AdaNoj14} it was shown that for any $p>1$ there is  $\alpha^*<0$ such that for $-N\sqrt{\omega}<\alpha<\alpha^*$ the  profile  $\mathbf{\Phi}^\alpha_{0}$ defined by \eqref{Phi_k} minimizes the action functional $\bb S_\omega$
  on the  Nehari manifold
 $$
 \mathcal N=\{\mathbf{V}\in \mc E(\Gamma)\setminus\{0\}: \bb I_\omega(\bb V)=0\}.
 $$
 Namely, the  profile $\mathbf{\Phi}^\alpha_{0}$ is the ground state  for the action $\bb S_\omega$ on the manifold $\mathcal N$.
In \cite{AdaNoj16} the authors showed that $\bb \Phi_0^\alpha$ is a local minimizer of the energy functional  $\bb E$ defined by \eqref{energy} among functions with equal fixed  mass.
\begin{remark}
Note that  $\mathbf{\Phi}^\alpha_{k}\in \mathcal{N}$ for all $k$. In \cite{AdaNoj14} it was proved that for $k\neq 0$ and $\alpha<0$ we have 
$\bb S_\omega (\mathbf{\Phi}^\alpha_{0})<\bb S_\omega(\mathbf{\Phi}^\alpha_{k})
<\bb S_\omega(\mathbf{\Phi}^\alpha_{k+1})$.
\end{remark}
   
Until now nothing is known about variational properties of  the profiles 
$\mathbf{\Phi}^\alpha_k$ for $\alpha>0$.  Anyway, one can easily verify that 
$\bb S_\omega (\mathbf{\Phi}^\alpha_{0})>\bb S_\omega(\mathbf{\Phi}^\alpha_{k})
>\bb S_\omega(\mathbf{\Phi}^\alpha_{k+1})$, \,$k\neq 0$.
 
We consider three minimization problems
\begin{equation} \label{def-deq} 
d_{\eq}(\omega)=\inf\{\bb S_\omega(\bb V):\, \bb V\in \EE_{\eq}(\Gamma)\setminus\{0\}, \, 
\bb I_\omega(\bb V)=0 \}, 
\end{equation}
   and
    \begin{equation}\label{well_27c}
        d^{\li}_{\rad}(\omega)=\inf\left\{\begin{array}{c}\tfrac{1}{2}\|v'\|_2^2+\tfrac{\omega}{2}\|v\|_2^2-\frac{1}{p+1}\|v\|_{p+1}^{p+1}+\tfrac{\alpha}{N}|v(0)|^2:\\ \|v'\|_2^2+\omega\|v\|_2^2-\|v\|_{p+1}^{p+1}+\tfrac{2\alpha}{N}|v(0)|^2=0, \,v\in H^1_{\rad}(\mathbb{R})\setminus\{0\}    
    \end{array}\right\}, 
        \end{equation}
         \begin{equation*}
        d^{\half}(\omega)=\inf\left\{\begin{array}{c}\tfrac{1}{2}\|v'\|_2^2+\tfrac{\omega}{2}\|v\|_2^2-\frac{1}{p+1}\|v\|_{p+1}^{p+1}+\tfrac{\alpha}{2N}|v(0)|^2:\\ \|v'\|_2^2+\omega\|v\|_2^2-\|v\|_{p+1}^{p+1}+\tfrac{\alpha}{N}|v(0)|^2=0, \,v\in H^1(\mathbb{R}_+)\setminus\{0\}    
    \end{array}\right\}. 
        \end{equation*}
It is easily seen that 
$$d_{\eq}(\omega)=N d^{\half}(\omega)=\tfrac{N}{2} d^{\li}_{\rad}(\omega).$$
From  the results by \cite{FukJea08, FukOht08} one gets
\begin{equation}\label{well_28} 
d_{\eq}(\omega)=\bb S_\omega(\bb\Phi_0^\alpha)=\tfrac{N}{2} d^{\li}_{\rad}(\omega)=\tfrac{N}{2}\left(\tfrac{1}{2}\|\phi'_\omega\|_2^2+\tfrac{\omega}{2}\|\phi_\omega\|_2^2-\tfrac{1}{p+1}\|\phi_\omega\|_{p+1}^{p+1}+\tfrac{\alpha}{N}|\phi_\omega(0)|^2\right),
\end{equation}
where
        $$\phi_\omega(x)=\left[\frac{p+1}{2}\omega\sech^2\left(\frac{p-1}{2}\sqrt{\omega}|x|-\tanh^{-1}\left(\frac{\alpha}{N\sqrt{\omega}}\right)\right)\right]^{\frac{1}{p-1}}.$$
Using \eqref{var_1}, we obtain the following useful formula
\begin{equation}\label{well_29}
d_{\eq}(\omega)=\bb S_\omega(\bb\Phi_0^\alpha)=\inf\left\{\frac{p-1}{2(p+1)}\|\bb V\|_{p+1}^{p+1}:\, \bb V\in \EE_{\eq}(\Gamma)\setminus\{0\}, \, 
\bb I_\omega(\bb V)=0 \right\}. 
\end{equation}
In the sequel for simplicity we will always use the notation $\bb\Phi(x):=\bb\Phi_0^\alpha(x).$

        \begin{remark}
        Note that in the case $\alpha=0$ one arrives at analogous result, that is
         \begin{equation*} d_{\eq}^0(\omega)=\bb S_\omega^0(\bb\Phi_0^0)=\tfrac{N}{2} d^{\li,0}_{\rad}(\omega)=\tfrac{N}{2}\left(\tfrac{1}{2}\|\phi'_{\omega,0}\|_2^2+\tfrac{\omega}{2}\|\phi_{\omega,0}\|_2^2-\tfrac{1}{p+1}\|\phi_{\omega,0}\|_{p+1}^{p+1}\right),\end{equation*}
        where
        $$\phi_{\omega,0}(x)=\left[\frac{p+1}{2}\omega\sech^2\left(\frac{p-1}{2}\sqrt{\omega}x\right)\right]^{\frac{1}{p-1}}, \, x\in\mathbb{R},\quad \bb\Phi_0^0(x)=(\phi_{\omega,0}(x))_{j=1}^N,\, x\in\mathbb{R_+},
      $$
      and $d_{\eq}^0(\omega),\bb S_\omega^0, d^{\li,0}_{\rad}(\omega)$ correspond to the case $\alpha=0$ in \eqref{well_27b},\eqref{well_27c},\eqref{well_28}.
        \end{remark}

\section{Proof of strong instability results}\label{sec4} 

\subsection{Proof of Theorem \ref{alpha>0}}
The proof of theorem relies on the following three lemmas. 
\begin{lemma}\label{lem-dsp01} Let the functional $\bb P(\bb V)$ be defined by \eqref{well_16a}.
If $\bb V\in \EE_{\eq} (\Gamma)\setminus\{0\}$ 
satisfies $\bb P(\bb V)\le 0$, then 
\[
d_{\eq}(\omega) 
\le \bb S_\omega(\bb V)-\frac{1}{2} \bb P(\bb V). 
\]
\end{lemma}

\begin{proof}
Let $\bb V\in \EE_{\eq} (\Gamma)\setminus\{0\}$ 
satisfy $\bb P(\bb V)\le 0$. 
Define 
$\bb V^{\lambda}(x)
=\lambda^{1/2} \bb V(\lambda x)$ for $\lambda>0$, 
and consider the function 
\[
(0,\infty)\ni \lambda \mapsto 
\bb I_{\omega} (\bb V^{\lambda})
=\lambda^2 \|\bb V'\|_2^2
+\alpha \lambda |v_1(0)|^2 
-\lambda^{\beta} \| \bb V \|_{p+1}^{p+1}
+\omega \| \bb V \|_2^2, 
\]
where we put $\beta=\frac{p-1}{2}\ge 2$. 
Then, we have 
\begin{equation} \label{fg01}
\lim_{\lambda\to +0} \bb I_{\omega} (\bb V^{\lambda})
=\omega \| \bb V \|_2^2>0, \quad 
\lim_{\lambda\to +\infty} \bb I_{\omega} (\bb V^{\lambda})=-\infty. 
\end{equation}
\iffalse

Here, we remark that 
by $\bb P(\bb V)\le 0$, we have
\[
\|\bb V'\|_2^2
\le \|\bb V'\|_2^2+\alpha  |v_1(0)|^2 
\le \frac{\beta}{p+1}\| \bb V \|_{p+1}^{p+1}, 
\]
so that when $\beta=2$, we have 
\[
\|\bb V'\|_2^2-\| \bb V \|_{p+1}^{p+1}
\le -\frac{p-1}{p+1} \| \bb V \|_{p+1}^{p+1}<0. 
\]
\fi

By \eqref{fg01}, there exists $\lambda_0\in (0,\infty)$ such that 
$\bb I_{\omega} (\bb V^{\lambda_0})=0$. 
Then, by  definition \eqref{def-deq}, we have 
$d_{\eq}(\omega) \le \bb S_\omega(\bb V^{\lambda_0})$. 

Moreover, since $\beta\ge 2$, the function 
\[
(0,\infty)\ni \lambda \mapsto 
\bb S_{\omega} (\bb V^{\lambda})-\frac{\lambda^2}{2} \bb P(\bb V)
=\frac{2\lambda-\lambda^2}{4} \alpha |v_1(0)|^2 
+\frac{\beta \lambda^2-2\lambda^{\beta}}{2(p+1)}\| \bb V \|_{p+1}^{p+1}
+\frac{\omega}{2} \| \bb V \|_2^2
\]
attains its maximum at $\lambda=1$.   Indeed, to show this it is sufficient to study the derivative of  the function $f(\lambda):=\bb S_{\omega} (\bb V^{\lambda})-\frac{\lambda^2}{2} \bb P(\bb V)$. 
Thus, by using $\bb P(\bb V)\le 0$, we have 
\[
d_{\eq}(\omega)
\le \bb S_\omega(\bb V^{\lambda_0})
\le \bb S_\omega(\bb V^{\lambda_0})-\frac{\lambda_0^2}{2} \bb P(\bb V)
\le \bb S_\omega(\bb V)-\frac{1}{2} \bb P(\bb V). 
\]
This completes the proof. 
\end{proof}

We introduce
\[
\mc B_{\omega}^{+}
:=\{\bb V\in \mc E_{\eq}(\Gamma):\,\bb S_\omega(\bb V)< d_{\eq}(\omega), \, 
\bb P(\bb V)<0 \}.
\]
Upper index $+$ means that we consider the case of positive $\alpha$.
\begin{lemma}\label{lem-dsp02}
The set $\mc B_{\omega}^+$ is invariant under the flow of \eqref{NLS_graph_ger}. 
That is,  if  $\bb U_0\in \mc B_{\omega}^{+}$, 
then the solution  $\bb U(t)$ to \eqref{NLS_graph_ger} with $\bb U(0)=\bb U_0$ 
belongs to $\mc B_{\omega}^+$  for all $t\in [0, T_{H^1})$. 
\end{lemma}

\begin{proof}
First, by \cite[Theorem 3.4]{AngGol18}, 
we have $\bb U(t)\in \mc E_{\eq}(\Gamma)$ for all  $t\in [0,T_{H^1})$. 
Further, 
by  conservation laws \eqref{conservation-laws}, 
for all $t\in [0,T_{H^1})$, we have 
\[
\bb S_{\omega}(\bb U(t))
=\bb E(\bb U(t))+\frac{\omega}{2} \|\bb U(t)\|_{2}^2
=\bb S_{\omega}(\bb U_0) <d_{\eq}(\omega). 
\]

Next, we prove that $\bb P(\bb U(t))<0$ for all $t\in [0,T_{H^1})$. 
Suppose that this were not true. 
Then, there exists $t_0\in (0,T_{H^1})$ such that $\bb P(\bb U(t_0))=0$. 
Moreover, since $\bb U(t_0)\ne 0$, it follows from Lemma \ref{lem-dsp01} that 
\[
d_{\eq}(\omega)\le \bb S_{\omega}(\bb U(t_0))-\frac{1}{2} \bb P(\bb U(t_0))
=\bb S_{\omega}(\bb U(t_0)).
\]
This contradicts the fact that 
$\bb S_{\omega}(\bb U(t))<d_{\eq}(\omega)$ for all $t\in [0,T_{H^1})$. 
Hence, we have $\bb P(\bb U(t))<0$ for all $t\in [0,T_{H^1})$. 
\end{proof}

\begin{lemma}\label{lem-dsp03}
If $\bb U_0\in \mc B_{\omega}^+\cap \Sigma (\Gamma)$, 
then the solution  $\bb U(t)$ to \eqref{NLS_graph_ger} with $\bb U(0)=\bb U_0$ 
blows up in finite time. 
\end{lemma}

\begin{proof}
By Lemma \ref{lem-dsp02} and Proposition \ref{prop-virial}, 
we have 
$\bb U(t)\in \mc B_{\omega}^+ \cap \Sigma(\Gamma)$ for all $t\in [0,T_{H^1})$. 
Moreover, by  virial identity \eqref{well_18}, 
 conservation laws \eqref{conservation-laws}
and Lemma \ref{lem-dsp01}, we have 
\begin{align*}
\frac{1}{16}\frac{d^2}{dt^2} \|x \bb U(t)\|_{2}^2 
=\frac{1}{2} \bb P(\bb U(t)) 
\le \bb S_{\omega} (\bb U(t))-d_{\eq} (\omega) 
=\bb S_{\omega}(\bb U_0)-d_{\eq}(\omega)<0
\end{align*}
for all $t\in [0,T_{H^1})$.
Denoting $-b:= \bb S_{\omega}(\bb U_0)-d_{\eq}(\omega)<0$ we get 
\[\|x \bb U(t)\|^2_2\leq -16bt^2+Ct+\|x\bb U_0\|_2^2,\]
from which we conclude $T_{H^1}<\infty$. 
\end{proof}

We are now in a position to give the proof of Theorem \ref{alpha>0}. 

\begin{proof}[Proof of Theorem \ref{alpha>0}]
First, we note that $\bb \Phi
=\bb\Phi_0^\alpha(x)
\in \EE_{\eq}(\Gamma)\cap \Sigma(\Gamma)$. 

Since 
$\bb S_{\omega}'(\bb \Phi)=0$ and $\beta=\frac{p-1}{2}\ge 2$, 
the function 
\begin{align*}
(0,\infty)\ni \lambda\mapsto 
\bb S_{\omega}(\bb \Phi^{\lambda})
=\frac{\lambda^2}{2} \|\bb \Phi'\|_{2}^2
+\frac{\alpha}{2} \lambda |\varphi(0)|^2
+\frac{\omega}{2} \|\bb \Phi\|_{2}^2 
-\frac{\lambda^{\beta}}{p+1} \|\bb \Phi\|_{p+1}^{p+1}
\end{align*}
attains its maximum at $\lambda=1$, 
and we see that 
\begin{align*}
\bb S_{\omega}(\bb \Phi^{\lambda})
<\bb S_{\omega}(\bb \Phi)=d_{\eq}(\omega), \quad 
\bb P (\bb \Phi^{\lambda})
=\lambda \partial_{\lambda} \bb S_{\omega} (\bb \Phi^{\lambda})
<0
\end{align*}
for all $\lambda>1$. 
Thus, for $\lambda>1$, 
$\bb \Phi^{\lambda} \in \mc B_{\omega}^+\cap \Sigma(\Gamma)$, 
and it follows from Lemma \ref{lem-dsp03} that 
the solution $\bb U(t)$ of \eqref{NLS_graph_ger} 
with $\bb U(0)=\bb \Phi^{\lambda}$ 
blows up in finite time. 
Finally, since
$\displaystyle{
\lim_{\lambda\to 1} \|\bb \Phi^{\lambda}
-\bb \Phi\|_{H^1}=0}$, 
the proof is completed. 
\end{proof}

\begin{remark} 
Observe that for $\alpha=0$ one can prove analogously the result:
\textit{ Let  $\alpha=0$,\, $\omega>0$, and $p\geq 5$, 
then  the standing wave $e^{i\omega t}\bb \Phi_0^0(x)$ is  strongly unstable.  
%in $\mc E(\Gamma)$.
 }
\end{remark}
       
\begin{remark}
\begin{itemize}
\item[$(i)$]\, In \cite{BerCaz81} the authors studied the strong instability of the  standing wave solution (ground state) to the NLS equation 
$$i\partial_tu=-\Delta u-|u|^{p-1}u,\quad (t,x)\in\mathbb{R}\times \mathbb{R}^n.$$
They have used the fact that the ground state  is the minimizer of the problem
$$d(\omega)=\inf\{S_\omega(v):\, v\in H^1(\mathbb{R}^n)\setminus\{0\},\, \,P(v)=0\},$$
where $S_\omega$ is the corresponding action functional, and $P$ is from the virial identity.
Similarly to the proof of Theorem \ref{alpha<0}, the authors use invariance of the set
$$\mc B_\omega=\{v\in H^1(\mathbb{R}^n):\, S_\omega(v)<d(\omega),\,\,\, P(v)<0 \}$$  under the flow of the NLS equation.
\item[$(ii)$]\, In \cite{CozFuk08} the authors considered the particular case $n=2$, i.e.  the  NLS-$\delta$  equation on the line. Namely, the strong instability of the standing wave $\varphi_{\omega,\gamma}$ was proved for $\gamma<0$ and $p\geq 5$. The authors used   the  fact that  $\varphi_{\omega,\gamma}$ is the minimizer of the problem
$$d_{\mc M}=\inf\{S_{\omega,\gamma}(v):\, H_{\rad}^1(\mathbb{R})\setminus\{0\}, \,\, P_\gamma(v)=0,\, \, I_{\omega,\gamma}(v)\leq 0\}.$$
Moreover, the invariance of the set 
$$\mc B_{\omega,\gamma}=\{v\in H^1_{\rad}(\mathbb{R}):\, S_{\omega,\gamma}(v)<S_{\omega,\gamma}(\varphi_{\omega,\gamma}),\,\, P_\gamma(v)<0,\,\, I_{\omega,\gamma}(v)< 0\}$$ under the flow of the NLS-$\delta$ equation was used. 
\item[$(iii)$]\, The proof by \cite{CozFuk08} mentioned above can be generalized to the case of $\Gamma$ and $\alpha>0$. Namely, one needs to prove that $\bb \Phi_0^\alpha$ is the minimizer of  
$$d_{\mc M}(\omega)=\inf\{\bb S_\omega(\bb V): \bb V\in \mc E_{\eq}(\Gamma)\setminus\{0\},\quad \bb P(\bb V)=0,\quad \bb I_\omega(\bb V)\leq 0\},$$
and to substitute  $\mc B_{\omega,\gamma}$ by 
$$\mc B_{\omega,\alpha}=\{\bb V\in \mc E_{\eq}(\Gamma):
\, \bb S_\omega(\bb V)<\bb S_\omega(\bb\Phi_0^\alpha),\quad  \bb I_\omega(\bb V)<0,\quad \bb P(\bb V)<0\}.$$  
\end{itemize}
\end{remark} 
 
\subsection{Proof of Theorem \ref{alpha<0}} 
As in the previous case the proof can be divided into series of lemmas. 
\begin{lemma}\label{second_der}
Let $\alpha<0, p>5$ and $\omega>\frac{\alpha^2}{N^2}.$ Let $\omega_1$ be the number defined in Theorem \ref{alpha<0}. Then $\partial_\lambda^2\bb E(\bb \Phi^\lambda)|_{\lambda=1}\leq 0$ if and only if $\omega\geq \omega_1$. 
\end{lemma}
\begin{proof}
Since $\bb P(\bb \Phi)=\|\bb \Phi'\|_2^2+\frac{\alpha}{2}|\varphi(0)|^2-\frac{p-1}{2(p+1)}\|\bb\Phi\|_{p+1}^{p+1}=0,$ the condition $\partial_\lambda^2\bb E(\bb \Phi^\lambda)|_{\lambda=1}= \|\bb \Phi'\|_2^2-\frac{(p-1)(p-3)}{4(p+1)}\|\bb\Phi\|_{p+1}^{p+1}\leq 0$ is equivalent to 
\begin{equation}\label{sec_1}
-\alpha|\varphi(0)|^2\leq \frac{(p-1)(p-5)}{2(p+1)}\|\bb\Phi\|_{p+1}^{p+1}.
\end{equation}
Denoting $\xi=\frac{-\alpha}{N\sqrt{\omega}},$ we obtain
\begin{equation}\label{sec_2}
|\varphi(0)|^2=\left[\frac{(p+1)\omega}{2}\sech^2(\tanh^{-1}\xi)\right]^{\frac{2}{p-1}}=\left[\frac{(p+1)\omega}{2}(1-\xi^2)\right]^{\frac{2}{p-1}},
\end{equation}
and
\begin{equation}\label{sec_3}
\begin{split}
&\|\bb\Phi\|_{p+1}^{p+1}=N\int\limits_{\mathbb{R}_+}\left[\frac{(p+1)\omega}{2}\sech^2\left(\frac{(p-1)\sqrt{\omega}}{2}x+\tanh^{-1}\xi\right)\right]^{\frac{p+1}{p-1}}dx\\&=\frac{2N}{(p-1)\sqrt{\omega}}\left(\frac{(p+1)\omega}{2}\right)^{\frac{p+1}{p-1}}\int\limits_{\tanh^{-1}\xi}^\infty(\sech^2y)^{\frac{p+1}{p-1}}dy\\&=\frac{2N}{(p-1)\sqrt{\omega}}\left(\frac{(p+1)\omega}{2}\right)^{\frac{p+1}{p-1}}\int\limits_{\xi}^1(1-s^2)^{\frac{2}{p-1}}ds.
\end{split}
\end{equation}
Using \eqref{sec_2} and \eqref{sec_3}, we see that \eqref{sec_1} is equivalent to 
\begin{equation}\label{sec_4}
\frac{p-5}{2}\int\limits_\xi^1(1-s^2)^{\frac{2}{p-1}}ds\geq\xi(1-\xi^2)^{\frac{2}{p-1}}.
\end{equation}
Consider the function $f(\xi)=\frac{p-5}{2}\int\limits_\xi^1(1-s^2)^{\frac{2}{p-1}}ds-\xi(1-\xi^2)^{\frac{2}{p-1}}, \,\, \xi\in[0,1].$
Observing that $f(0)>1, f(1)=0$, the derivative $f'(\xi)$ has a unique zero in $(0,1)$, and $f'(\xi)<0$ for small positive $\xi$, the function $f$ has a unique zero $\xi_1$ in $(0,1)$. Hence $f(\xi)\geq 0$ for $\xi\in[0, \xi_1]$, and  therefore,  recalling that $\xi=\frac{-\alpha}{N\sqrt{\omega}},$ inequality \eqref{sec_4} holds for  $\omega\geq\omega_1=\frac{\alpha^2}{N^2\xi^2_1}.$ 
\end{proof} 
Throughout this Section we impose the assumption $\omega\geq \omega_1$ or equivalently, by the above Lemma, we assume that $\partial_\lambda^2\bb E(\bb \Phi^\lambda)|_{\lambda=1}\leq 0$.
\begin{lemma} \label{lem-vp00}
If $\bb V\in \EE_{\eq}(\Gamma)$ and $\| \bb V\|_{p+1}=\|\bb\Phi\|_{p+1}$, 
then $\bb S_{\omega}(\bb V)\ge d_{\eq}(\omega)$.
\end{lemma}

\begin{proof}
First, we prove $\bb I_{\omega}(\bb V)\ge 0$ by contradiction. 
Suppose that $\bb I_{\omega}(\bb V)<0$. 
Let
\[
\lambda_1=\left(
\frac{\| \bb V' \|_2^2+\omega \|\bb V \|_2^2+\alpha |v_1(0)|^2}
{\| \bb V \|_{p+1}^{p+1}}\right)^{1/(p-1)}. 
\]
Then, $0<\lambda_1<1$ and $\bb I_{\omega}(\lambda_1 \bb V)=0$. 
Moreover, since $\lambda_1 \bb V\in \EE_{\eq}(\Gamma)\setminus\{0\}$, 
it follows from \eqref{well_29} and \eqref{var_1} that 
\begin{align*}
\frac{p-1}{2(p+1)} \| \bb \Phi\|_{p+1}^{p+1}
&=d_{\eq}(\omega)
\le \bb S_{\omega}(\lambda_1 \bb V)
=\bb S_{\omega}(\lambda_1 \bb V)-\frac{1}{2} \bb I_{\omega}(\lambda_1 \bb V) \\
&=\frac{p-1}{2(p+1)} \| \lambda_1 \bb V\|_{p+1}^{p+1}
<\frac{p-1}{2(p+1)} \| \bb V\|_{p+1}^{p+1}.
\end{align*}
This contradicts the assumption $\| \bb V\|_{p+1}=\| \bb \Phi\|_{p+1}$. 
Thus, we have $\bb I_{\omega}(\bb V)\ge 0$. 

Finally, we arrive at
\[
d_{\eq}(\omega)=\frac{p-1}{2(p+1)} \| \bb \Phi\|_{p+1}^{p+1}
\le \frac{p-1}{2(p+1)} \| \bb V\|_{p+1}^{p+1}+\frac{1}{2} \bb I_{\omega}(\bb V)
=\bb S_{\omega}(\bb V). 
\]
This completes the proof. 
\end{proof}

 \iffalse 

First, we remark the following. 
For $v\in H^1(\mathbb{R}_+)$, we define 
\begin{align*}
S_{\omega}^{\half}(v)
&=\int_{\mathbb{R}_{+}}
\left(
\frac{1}{2}|v'(x)|^2
+\frac{\omega}{2} |v(x)|^2
-\frac{1}{p+1} |v(x)|^{p+1}
\right) dx
+\frac{\alpha}{2N}|v(0)|^2, \\
P^{\half}(v)
&=\int_{\mathbb{R}_{+}}
\left(
|v'(x)|^2
-\frac{\beta}{p+1} |v(x)|^{p+1}
\right) dx
+\frac{\alpha}{2N}|v(0)|^2, \quad
\beta=\frac{p-1}{2}. 
\end{align*}
Then, for $\bb V=(v_j)_{j=1}^{N}\in \EE_{\eq}(\Gamma)$, we have 
\[
\bb S_{\omega}(\bb V)=N S_{\omega}^{\half}(v_1), \quad 
\bb P(\bb V)=N P^{\half} (v_1),
\]

\fi
\begin{lemma}\label{lem-dsp21} %Assume $\partial^2_\lambda\bb E(\bb \Phi^\lambda)|_{\lambda=1}\leq 0$.
If $\bb V\in \EE_{\eq}(\Gamma)$ satisfies 
\[
\|\bb V\|_2\le \|\bb \Phi\|_2, \quad 
\|\bb V\|_{p+1}>\|\bb \Phi\|_{p+1}, \quad 
\bb P(\bb V)\le 0,
\]
then
\[
d_{\eq}(\omega)\le \bb S_{\omega}(\bb V)-\frac{1}{2}\bb P(\bb V). 
\]
\end{lemma}

\begin{proof}
\iffalse 
By the assumption, we have
\[
\|v_1\|_2\le \|\varphi_{0,\omega}\|_2, \quad 
\|v_1\|_{p+1}> \|\varphi_{0,\omega}\|_{p+1}, \quad 
P^{\half}(v_1)\le 0. 
\]

Then, by Lemma 3.1 of \cite{Oht18} 
(see also Lemma 3.2 of \cite{FukayaO}), 
we have 
\[
S_{\omega}^{\half}(\varphi_{0,\omega})
\le S_{\omega}^{\half}(v_1)-\frac{1}{2}P^{\half}(v_1). 
\]

Thus, we have 
\[
d_{\eq}(\omega)=\bb S_{\omega}(\bb \Phi_{0,\omega})
=NS_{\omega}^{\half}(\varphi_{0,\omega})
\le N S_{\omega}^{\half}(v_1)-\frac{N}{2}P^{\half}(v_1)
=\bb S_{\omega}(\bb V)-\frac{1}{2}\bb P(\bb V).
\]

This completes the proof. 
\fi

Define \[\lambda_0=\left(\frac{\|\bb \Phi\|_{p+1}^{p+1}}{\|\bb V\|_{p+1}^{p+1}}\right)^\frac{2}{p-1},\]
then $0<\lambda_0<1$, moreover, $\|\bb V^{\lambda_0}\|_{p+1}^{p+1}=\lambda_0^{\frac{p-1}{2}}\|\bb V\|_{p+1}^{p+1}=\|\bb \Phi\|_{p+1}^{p+1}$. 

 The key ingredient  of the proof is the inequality $\bb S_\omega(\bb \Phi)\leq \bb S_\omega(\bb V^{\lambda_0})$. It follows by Lemma \ref{lem-vp00} since $d_{\eq}(\omega)=\bb S_\omega(\bb \Phi)$ and $\|\bb \Phi\|_{p+1}=\|\bb V^{\lambda_0}\|_{p+1}$. 

Define $f(\lambda)=\bb S_\omega(\bb V^\lambda)-\frac{\lambda^2}{2}\bb P(\bb V),\,\,\lambda\in(0,1].$ Suppose that $f(\lambda_0)\leq f(1)$. Using  $P(\bb V)\leq 0$, one gets
\[\bb S_\omega(\bb \Phi)\leq \bb S_\omega(\bb V^{\lambda_0})\leq \bb S_\omega(\bb V^{\lambda_0})-\frac{\lambda_0^2}{2}\bb P(\bb V)\leq \bb S_\omega(\bb V)-\frac{1}{2}\bb P(\bb V), \] and we are done. Thus, it is sufficient to prove $f(\lambda_0)\leq f(1).$ The proof is analogous to the proofs of \cite[Lemma 3.2]{FukayaO} and  \cite[Lemma 3.1]{Oht18}.
Denote $\beta=\frac{p-1}{2}.$ Observe that 
\begin{equation}\label{lem-dsp22}
f(\lambda_0)\leq f(1)\,\,\Longleftrightarrow\,\, -\alpha|v_1(0)|^2\leq \frac{2}{p+1}\frac{2\lambda_0^\beta-\beta\lambda_0^2-2+\beta}{(\lambda_0-1)^2}\|\bb V\|_{p+1}^{p+1}.
\end{equation}
Thus, one should be aimed to prove the second inequality in  \eqref{lem-dsp22}.
Note that  the condition $\partial^2_\lambda\bb E(\bb \Phi^\lambda)|_{\lambda=1}=\|\bb \Phi'\|_2^2-\frac{(p-1)(p-3)}{4(p+1)}\|\bb \Phi\|_{p+1}^{p+1}\leq 0$ is equivalent to 
\begin{equation}\label{lem-dsp23}
\|\bb \Phi'\|_2^2\leq \frac{\beta (\beta-1)}{p+1}\|\bb \Phi\|_{p+1}^{p+1}.
\end{equation}
Using Pohozaev-type equality 
\[\|\bb \Phi'\|_2^2-\omega\|\bb \Phi\|_2^2+\frac{2}{p+1}\|\bb \Phi\|_{p+1}^{p+1}=0\] and estimate \eqref{lem-dsp23}, we deduce
\begin{equation}\label{lem-dsp24}
\omega\|\bb \Phi\|_2^2=\|\bb \Phi'\|_2^2+\frac{2}{p+1}\|\bb \Phi\|_{p+1}^{p+1}\leq \frac{\beta^2-\beta+2}{p+1}\|\bb \Phi\|_{p+1}^{p+1}.
\end{equation}  
Combining $\|\bb V\|_2^2\leq \|\bb \Phi\|_2^2$ and $\|\bb \Phi\|_{p+1}^{p+1}=\lambda_0^\beta\|\bb V\|_{p+1}^{p+1}$, we obtain from \eqref{lem-dsp24}
\begin{equation}\label{lem-dsp24a}
\omega\|\bb \Phi\|_2^2\leq \frac{\beta^2-\beta+2}{p+1}\lambda^\beta_0\|\bb V\|_{p+1}^{p+1}.
\end{equation}
By the proof of Lemma \ref{lem-vp00}, we have
\[\bb I_\omega(\bb V^{\lambda_0})=\lambda_0^2\|\bb V'\|_2^2+\omega\|\bb V\|_2^2+\lambda_0\alpha|v_1(0)|^2-\lambda^\beta_0\|\bb V\|_{p+1}^{p+1}\geq 0, \] and therefore
\begin{equation}\label{lem-dsp25}
-\lambda_0\alpha|v_1(0)|^2\leq \lambda_0^2\|\bb V'\|_2^2+\omega\|\bb V\|_2^2-\lambda^\beta_0\|\bb V\|_{p+1}^{p+1}.
\end{equation}
The condition $\bb P(\bb V)=\|\bb V'\|_2^2+\frac{\alpha}{2}|v_1(0)|^2-\frac{\beta}{p+1}\|\bb V\|_{p+1}^{p+1}\leq 0$  implies
\begin{equation}\label{lem-dsp26}
\|\bb V'\|_2^2\leq -\frac{\alpha}{2}|v_1(0)|^2+\frac{\beta}{p+1}\|\bb V\|_{p+1}^{p+1}.
\end{equation}
Combining \eqref{lem-dsp24a}-\eqref{lem-dsp26} we get
\begin{equation}\label{lem-dsp27}
-\alpha|v_1(0)|^2\leq \frac{2}{p+1}\frac{\beta(\lambda^2_0+(\beta-3)\lambda^\beta_0)}{\lambda_0(2-\lambda_0)}\|\bb V\|_{p+1}^{p+1}.
\end{equation}
By \eqref{lem-dsp22} and \eqref{lem-dsp27}, we conclude that  $f(\lambda_0)\leq f(1)$ holds if 
\begin{equation}\label{lem-dsp28}
\frac{\beta(\lambda^2+(\beta-3)\lambda^\beta)}{\lambda(2-\lambda)}\leq \frac{2\lambda^\beta-\beta\lambda^2-2+\beta}{(\lambda-1)^2}\,\,\text{for}\,\,\lambda\in(0,1).
\end{equation}
Inequality \eqref{lem-dsp28} can be verified  by proving that the derivative of the function 
$$g(\lambda)=\frac{\beta(\lambda^2+(\beta-3)\lambda^\beta)}{\lambda(2-\lambda)}-\frac{2\lambda^\beta-\beta\lambda^2-2+\beta}{(\lambda-1)^2}$$ is nonpositive for $\lambda\in(0,1)$. This can be done similarly to the second part of the proof of \cite[Lemma 3.2]{FukayaO}.
 \end{proof}
 \begin{remark}
 Observe that the condition $\partial^2_\lambda \bb E(\bb \Phi^\lambda)|_{\lambda=1}\leq 0$ is crucial for the proof of the key inequality  $d_{\eq}(\omega)\leq \bb S_\omega(\bb V)-\frac{1}{2}\bb P(\bb V).$
 \end{remark}
We introduce
\begin{align*}
\mc B_{\omega}^-
:=\left\{\begin{array}{c}
 \bb V\in \mc E_{\eq}(\Gamma):\,\bb S_\omega(\bb V)< d_{\eq}(\omega), \, 
\bb P(\bb V)<0, \\
\| \bb V\|_2\le \| \bb \Phi \|_2,\, 
\| \bb V\|_{p+1}> \| \bb \Phi \|_{p+1}\end{array}\right\}.
\end{align*}

\begin{lemma}\label{lem-inv}
The set $\mc B_{\omega}^-$ is invariant under the flow of \eqref{NLS_graph_ger}. 
That is,  if  $\bb U_0\in \mc B_{\omega}^-$, 
then the solution  $\bb U(t)$ to \eqref{NLS_graph_ger} with $\bb U(0)=\bb U_0$ 
belongs to $\mc B_{\omega}^-$  for all $t\in [0, T_{H^1})$. 
\end{lemma}

\begin{proof}
First, by \cite[Theorem 3.4]{AngGol18}, 
we have $\bb U(t)\in \mc E_{\eq}(\Gamma)$ for all  $t\in [0,T_{H^1})$. 
Further, 
by  conservation laws \eqref{conservation-laws}, 
for all $t\in [0,T_{H^1})$, we have 
\[
\bb S_{\omega}(\bb U(t))
=\bb E(\bb U(t))+\frac{\omega}{2} \|\bb U(t)\|_{2}^2
=\bb S_{\omega}(\bb U_0) <d_{\eq}(\omega), \quad 
\| \bb U(t) \|_2=\| \bb U_0 \|_2\ge \| \bb \Phi\|_2. 
\]

Next, we prove that $\bb P(\bb U(t))<0$ for all $t\in [0,T_{H^1})$. 
Suppose that this were not true. 
Then, there exists $t_0\in (0,T_{H^1})$ such that $\bb P(\bb U(t_0))=0$. 
Moreover, since $\bb U(t_0)\ne 0$, it follows from Lemma \ref{lem-dsp21} that 
\[
d_{\eq}(\omega)\le \bb S_{\omega}(\bb U(t_0))-\frac{1}{2} \bb P(\bb U(t_0))
=\bb S_{\omega}(\bb U(t_0)).
\]
This contradicts the fact that 
$\bb S_{\omega}(\bb U(t))<d_{\eq}(\omega)$ for all $t\in [0,T_{H^1})$. 
Thus, we have $\bb P(\bb U(t))<0$ for all $t\in [0,T_{H^1})$. 

Finally, we prove that 
$\| \bb U(t)\|_{p+1}> \| \bb \Phi \|_{p+1}$ for all $t\in [0,T_{H^1})$. 
Again suppose that this were not true. 
Then, there exists $t_1\in (0,T_{H^1})$ such that 
$\| \bb U(t_1)\|_{p+1}=\| \bb \Phi\|_{p+1}$. 
By Lemma \ref{lem-vp00}, we have  
$d_{\eq}(\omega) \le \bb S_{\omega}(\bb U(t_1))$. 
This contradicts the fact that 
$\bb S_{\omega}(\bb U(t))<d_{\eq}(\omega)$ for all $t\in [0,T_{H^1})$. 
Hence, we have 
$\| \bb U(t)\|_{p+1}> \| \bb \Phi \|_{p+1}$
for all $t\in [0,T_{H^1})$. 
\end{proof}
\begin{lemma}\label{blow_up_neg}
If $\bb U_0\in \mc B_{\omega}^- \cap \Sigma (\Gamma)$, 
then the solution  $\bb U(t)$ to \eqref{NLS_graph_ger} with $\bb U(0)=\bb U_0$ 
blows up in finite time. 
\end{lemma}

\begin{proof}
By Lemma \ref{lem-inv} and Proposition \ref{prop-virial}, 
we have 
$\bb U(t)\in \mc B_{\omega}^- \cap \Sigma(\Gamma)$ for all $t\in [0,T_{H^1})$. 
Moreover, by  virial identity \eqref{well_18}, 
 conservation laws \eqref{conservation-laws}
and Lemma \ref{lem-dsp21}, we have 
\begin{align*}
\frac{1}{16}\frac{d^2}{dt^2} \|x \bb U(t)\|_{2}^2 
=\frac{1}{2} \bb P(\bb U(t)) 
\le \bb S_{\omega} (\bb U(t))-d_{\eq} (\omega) 
=\bb S_{\omega}(\bb U_0)-d_{\eq}(\omega)<0
\end{align*}
for all $t\in [0,T_{H^1})$, 
from which we conclude $T_{H^1}<\infty$. 
\end{proof}

Finally, we give the proof of Theorem \ref{alpha<0}. 

\begin{proof}[Proof of Theorem \ref{alpha<0}]
First, we note that $\bb \Phi
=\bb\Phi_0^\alpha(x)
\in \EE_{\eq}(\Gamma)\cap \Sigma(\Gamma)$. 
Let $\omega\geq \omega_1,$ then, by Lemma \ref{second_der},   $\partial^2_\lambda\bb E(\bb \Phi^\lambda)|_{\lambda=1}\leq 0$.

Since 
$\bb S_{\omega}'(\bb \Phi)=0$ and $\beta=\frac{p-1}{2}>2$, 
the function 
\begin{align*}
(0,\infty)\ni \lambda\mapsto 
\bb S_{\omega}(\bb \Phi^{\lambda})
=\frac{\lambda^2}{2} \| \bb \Phi'\|_{2}^2
+\frac{\alpha}{2} \lambda |\varphi(0)|^2
+\frac{\omega}{2} \|\bb \Phi\|_{2}^2 
-\frac{\lambda^{\beta}}{p+1} \|\bb \Phi\|_{p+1}^{p+1}
\end{align*}
attains its maximum at $\lambda=1$, 
and we see that 
\begin{align*}
&\bb S_{\omega}(\bb \Phi^{\lambda})
<\bb S_{\omega}(\bb \Phi)=d_{\eq}(\omega), \quad 
\bb P (\bb \Phi^{\lambda})
=\lambda \partial_{\lambda} \bb S_{\omega} (\bb \Phi^{\lambda})<0, \\
&\| \bb \Phi^{\lambda}\|_2=\| \bb \Phi \|_2, \quad 
\| \bb \Phi^{\lambda}\|_{p+1}
=\lambda^{\beta} \| \bb \Phi\|_{p+1} 
>\| \bb \Phi \|_{p+1} 
\end{align*}
for all $\lambda>1$. 
Thus, for $\lambda>1$, 
$\bb \Phi^{\lambda} \in \mc B_{\omega}^- \cap \Sigma(\Gamma)$, 
and it follows from Lemma \ref{blow_up_neg} that 
the solution $\bb U(t)$ of \eqref{NLS_graph_ger} 
with $\bb U(0)=\bb \Phi^{\lambda}$ 
blows up in finite time. 

Finally, since
$\displaystyle{
\lim_{\lambda\to 1} \|\bb \Phi^{\lambda}
-\bb \Phi\|_{H^1}=0}$, 
the proof is completed. 
\end{proof}
\begin{remark}
In \cite{OhtYam16} the authors considered  the strong instability of the standing wave $\varphi_{\omega,\gamma}$ to the NLS-$\delta$ equation on the line for $\gamma>0, p>5$. It particular, it was shown that the condition $E(\varphi_{\omega,\gamma})>0$ guarantees strong instability of $\varphi_{\omega,\gamma}$. Here $E$ is the corresponding   energy functional. The proof by \cite{OhtYam16} can be easily adapted to the case of the NLS-$\delta$ equation on  $\Gamma$, that is, the condition $\bb E(\bb \Phi)>0$ guarantees the strong instability of $\bb \Phi$ for $\alpha<0, p>5$.

In \cite{Oht18} it was noted that the condition  $\bb E(\bb \Phi)>0$ implies $\partial^2_\lambda\bb E(\bb \Phi^\lambda)|_{\lambda=1}\leq 0$, and therefore Theorem \ref{alpha<0} is slightly better than an analogous result with the condition  $\bb E(\bb \Phi)>0$.
\end{remark}

\section{NLS-$\delta'$ equation on the line}\label{sec5}
In this section we consider strong instability of the standing wave solution $u(t,x)=e^{i\omega t}\varphi(x)$ to the NLS-$\delta'$ equation on the line
\begin{equation}\label{NLS_delta'}
i\partial_t u(t,x)-H_\gamma u(t,x) +|u|^{p-1}u=0,
\end{equation}
where $u(t,x): \mathbb{R}\times \mathbb{R}\rightarrow \mathbb{C}$, and $H_\gamma$ is the self-adjoint operator on $L^2(\mathbb{R})$ defined  by
\begin{equation*}
\begin{split}
(H_\gamma v)(x)&=-v''(x),\quad x\neq 0,\\
\dom(H_\gamma)&=\left\{v\in H^2(\mathbb{R}\setminus\{0\}): v'(0-)=v'(0+),\,\, v(0+)-v(0-)=-\gamma v'(0)\right\}.
\end{split}
\end{equation*}
The corresponding stationary equation has the form
\begin{equation}\label{stat_NLS_delta'}
H_\gamma\varphi+\omega\varphi-|\varphi|^{p-1}\varphi=0.
\end{equation}
From \cite[Proposition 5.1]{AdaNoj13} it follows that for $\gamma>0$  two functions below  (odd  and asymmetric) are the solutions to \eqref{stat_NLS_delta'}.
 \begin{equation}\label{odd}
 \varphi_{\omega,\gamma}^{odd}(x)=\sign(x) \left[\frac{(p+1)\omega}{2} \sech^2\left(\frac{(p-1)\sqrt{\omega}}{2}(|x|+y_0)\right)\right]^{\frac{1}{p-1}},\quad\hbox{$x\neq 0$;} \quad \tfrac4{\gamma^2} <\omega,
 \end{equation}
  \begin{equation}\label{asymm}
  \varphi_{\omega,\gamma}^{as}(x)= \left\{
                    \begin{array}{ll}
                      \Big[\frac{(p+1)\omega}{2} \sech^2\Big(\frac{(p-1)\sqrt{\omega}}{2}(x+y_1)\Big)\Big]^{\frac{1}{p-1}}, &\quad \hbox{$x>0;$} \\
                     - \Big[\frac{(p+1)\omega}{2} \sech^2\Big(\frac{(p-1)\sqrt{\omega}}{2}(x-y_2)\Big)\Big]^{\frac{1}{p-1}}, &\quad \hbox{$x<0$,}
                    \end{array}
                  \right., \quad \omega>\tfrac4{\gamma^2}\tfrac{p+1}{p-1},
                   \end{equation}   
where  $y_0=\frac{2}{(p-1)\sqrt{\omega}}\tanh^{-1}(\frac{2}{\gamma\sqrt{\omega}})$ and $y_j=\frac{2}{(p-1)\sqrt{\omega}}\tanh^{-1}(t_j),\, j\in\{1,2\}$. Here $0<t_1<t_2$ are constants satisfying the  system (see formula  (5.2) in \cite{AdaNoj13}):
\begin{equation}\label{t1t2}
\left\{\begin{array}{c} t_1^{p-1}-t_1^{p+1}=t_2^{p-1}-t_2^{p+1},\\
t_1^{-1}+t_2^{-1}=\gamma\sqrt{\omega}.
\end{array}\right.
\end{equation}
 Note that when transposing $y_1$ and $y_2$ in \eqref{asymm}, one gets the second asymmetric solution to \eqref{NLS_delta'}.
In \cite[Theorem 5.3]{AdaNoj13} it had been proven that $ \varphi_{\omega,\gamma}^{odd}(x)$ and $ \varphi_{\omega,\gamma}^{as}(x)$ are the  minimizers (for $\tfrac4{\gamma^2} <\omega\leq \tfrac4{\gamma^2}\tfrac{p+1}{p-1}$ and $\omega> \tfrac4{\gamma^2}\tfrac{p+1}{p-1}$ respectively) of the problem
\begin{equation*}
d_\gamma(\omega)=\inf\{S_{\omega,\gamma}(v):\, v\in H^1(\mathbb{R}\setminus\{0\})\setminus\{0\}, \,\,I_{\omega,\gamma}(v)=0\},
\end{equation*}
where 
\begin{equation*}
S_{\omega,\gamma}(v)=\tfrac{1}{2}\|v'\|_2^2+\tfrac{\omega}{2}\|v\|_2^2-\tfrac{1}{p+1}\|v\|_{p+1}^{p+1}-\tfrac{1}{2\gamma}|v(0+)-v(0-)|^2,
\end{equation*} and
\begin{equation*}
I_{\omega,\gamma}(v)=\|v'\|_2^2+\omega\|v\|_2^2-\|v\|_{p+1}^{p+1}-\tfrac{1}{\gamma}|v(0+)-v(0-)|^2.
\end{equation*}
Moreover, for $\omega>\tfrac4{\gamma^2}$  the odd profile $ \varphi_{\omega,\gamma}^{odd}$ is the minimizer of the problem (see the proof of Theorem 6.13 in \cite{AdaNoj13}) 
\begin{equation*}
d_{\gamma,\odd}(\omega)=\inf\{S_{\omega,\gamma}(v):\, v\in H^1_{\odd}(\mathbb{R}\setminus\{0\})\setminus\{0\},\,\, I_{\omega,\gamma}(v)=0\}.
\end{equation*}
The well-posedness result (in $H^1(\mathbb{R}\setminus\{0\})$) analogous to Theorem \ref{well_H1} was affirmed in \cite[Proposition 3.3 and 3.4]{AdaNoj13}.  Namely, the next proposition holds.
\begin{proposition}
Let $p > 1$. Then for any $u_0\in H^1(\mathbb{R}\setminus\{0\})$ there exists $T >
0$ such that equation \eqref{NLS_delta'} has a unique solution $u(t)\in C\left([0, T],H^1(\mathbb{R}\setminus\{0\})\right)\cap C^1\left([0, T],H^{-1}(\mathbb{R}\setminus\{0\})\right)$ satisfying $u(0) = u_0$. For each $T_0\in (0, T)$ the mapping
$ u_0\in H^1(\mathbb{R}\setminus\{0\})\mapsto u(t)\in C\left([0, T_0],H^1(\mathbb{R}\setminus\{0\})\right)$ is continuous.  Moreover, equation \eqref{NLS_delta'} has a maximal solution defined on
an interval of the form  $[0, T_{H^1})$, and the following "blow-up alternative" holds: either $T_{H^1} = \infty$ or $T_{H^1}<\infty$ and
$$\lim\limits_{t\to T_{H^1}}\|u(t)\|_{H^1(\mathbb{R}\setminus\{0\})} =\infty.$$
Furthermore, the charge and  the energy are conserved
$$ E_\gamma(u(t))= E_\gamma(u_0),\quad \|u(t)\|_2^2=\|u_0\|_2^2$$ for all $t\in[0, T_{H^1})$, where the energy is defined by
\begin{equation*}\label{energy_delta'} E_\gamma(v)=\frac{1}{2}\|v'\|_2^2-\frac{1}{2\gamma}|v(0+)-v(0-)|^2-\frac{1}{p+1}\|v\|_{p+1}^{p+1}.\end{equation*}
\end{proposition}
\begin{remark}
The well-posedness in $H^1_{\odd}(\mathbb{R}\setminus\{0\})$ was shown in the proof of \cite[Theorem 6.11]{AdaNoj13} using the explicit form of the integral kernel for the unitary group $e^{-iH_\gamma t}$.
\end{remark}
  Observing that $\inf\sigma(H_\gamma)=\left\{\begin{array}{c}
-\frac{4}{\gamma^2},\,\, \gamma<0\\
0,\,\,\, \gamma\geq 0,
\end{array}\right.$ and repeating the proof of Theorem \ref{well_D_H}  and Proposition \ref{prop-virial}, one gets the well-posedness in $D_{H_\gamma}$ and the following virial identity  for the solution $u(t)$ to the Cauchy problem with the initial data $u_0\in H^1(\mathbb{R}\setminus\{0\})\cap L^2(\mathbb{R}, x^2dx)$
\begin{equation}\label{virial_delta'}
\frac{d^2}{dt^2}\|xu(t)\|^2_2=8P_\gamma(u(t)), \quad t\in[0,T_{H^1}).
\end{equation}
Here
\begin{equation*}
P_\gamma(v)=\|v'\|^2_2-\tfrac{1}{2\gamma}|v(0+)-v(0-)|^2-\tfrac{p-1}{2(p+1)}\|v\|_{p+1}^{p+1},\quad v\in H^1(\mathbb{R}\setminus\{0\}).
\end{equation*} 
\begin{remark}
Observe that Strichartz estimates for $e^{-iH_\gamma t}$ analogous to estimates from \cite[Theorem 1.3]{BanIgn14} might be obtained using the explicit formula (3.6) in \cite{AlbBrz95}. In particular, the case of $A_+=0>A_-$ takes place in formula (3.6).
\end{remark}
Equality \eqref{virial_delta'} is the key ingredient of the proof of subsequent strong instability results.
\begin{theorem}\label{main_as}
Let $\gamma>0,\, p>5.$ There exists $\omega_2>\tfrac4{\gamma^2}\tfrac{p+1}{p-1}$ such that  $e^{i\omega t} \varphi_{\omega,\gamma}^{as}(x)$ is strongly unstable in $H^1(\mathbb{R}\setminus\{0\})$ for $\omega\geq\omega_2.$
\end{theorem}
  \begin{theorem}\label{main_odd}
Let $\gamma>0,\, p>5,\omega>\tfrac4{\gamma^2}\tfrac{p+1}{p-1}.$ Let $\xi_3(p)\in(0,1)$ be a unique solution of 
$$\frac{p-5}{2}\int\limits_\xi^1(1-s^2)^{\frac{2}{p-1}}ds=\xi(1-\xi^2)^{\frac{2}{p-1}},\quad (0<\xi<1),$$ and define $\omega_3=\omega_3(p,\gamma)=\frac{4}{\gamma^2\xi_3^2(p)}$. Then the standing wave solution $e^{i\omega t} \varphi_{\omega,\gamma}^{odd}(x)$ is strongly unstable for all $\omega\in [\omega_3, \infty).$\end{theorem}
\begin{remark}
Observe that $\omega_3>\tfrac4{\gamma^2}\tfrac{p+1}{p-1}$ since by \cite[Proposition 6.11]{AdaNoj13}  $e^{i\omega t} \varphi_{\omega,\gamma}^{odd}(x)$ is orbitally stable for $\tfrac4{\gamma^2}<\omega<\tfrac4{\gamma^2}\tfrac{p+1}{p-1}.$
\end{remark}

\textit{Key steps of the proofs of Theorem \ref{main_as} and \ref{main_odd}}.
Basically one needs to repeat the proof of Theorem \ref{alpha<0}. The only step which should be checked  carefully    is Lemma \ref{second_der}.

\textbf{1.} Consider the case of $\varphi_{\omega,\gamma}^{as}(x)$. Denote $\varphi_\gamma:=\varphi_{\omega,\gamma}^{as}$.   We need to show that $\partial_\lambda^2 E_\gamma(\varphi_\gamma^\lambda)|_{\lambda=1}\leq 0$ for $\omega\in[\omega_2,\infty),$ where $\omega_2$ is sufficiently large.  Using, $P_\gamma(\varphi_\gamma)=0$, it is easily seen that the condition  $\partial_\lambda^2 E_\gamma(\varphi_\gamma^\lambda)|_{\lambda=1}\leq 0$ is equivalent to 
\begin{equation}\label{delta'1}
\frac{1}{\gamma}|\varphi_\gamma(0+)-\varphi_\gamma(0-)|^2<\frac{(p-5)(p-1)}{2(p+1)}\|\varphi_\gamma\|_{p+1}^{p+1}.
\end{equation}
From \eqref{asymm}  and \eqref{t1t2} one gets
\begin{equation}\label{delta'2}
|\varphi_\gamma(0+)-\varphi_\gamma(0-)|^2=\left(\frac{p+1}{2}\omega\right)^{\tfrac 2{p-1}}\left((1-t_1^2)^{\tfrac 1{p-1}}+(1-t_2^2)^{\tfrac 1{p-1}}\right)^2,
\end{equation}
and 
\begin{equation}\label{delta'3}
\|\varphi_\gamma\|^{p+1}_{p+1}=\frac{2}{(p-1)\sqrt{\omega}}\left(\frac{p+1}{2}\omega\right)^{\tfrac{p+1}{p-1}}\left[\int\limits_{t_1}^1(1-s^2)^{\tfrac{2}{p-1}}ds+\int\limits_{t_2}^1(1-s^2)^{\tfrac{2}{p-1}}ds\right].
\end{equation}
Combining \eqref{delta'2} and \eqref{delta'3}, we deduce from \eqref{delta'1} that the condition $\partial_\lambda^2 E_\gamma(\varphi_\gamma^\lambda)|_{\lambda=1}\leq 0$ is equivalent to
\begin{equation}\label{delta'4}
\frac{p-5}{2}\left[\int\limits_{t_1}^1(1-s^2)^{\tfrac{2}{p-1}}ds+\int\limits_{t_2}^1(1-s^2)^{\tfrac{2}{p-1}}ds\right]-\frac{1}{\beta\sqrt{\omega}}\left[(1-t_1^2)^{\tfrac 1{p-1}}+(1-t_2^2)^{\tfrac 1{p-1}}\right]^2>0.
\end{equation}
Observe that $t_1$ and $t_2$ have the following asymptotics as $\omega \to \infty$ (see formula (6.34) in \cite{AdaNoj13})
\begin{equation*}
t_1=\frac{1}{\gamma\sqrt{\omega}}+o(\omega^{-\tfrac 1{2}}), \quad t_2=1-\frac{1}{2\gamma^{p-1}\omega^{\tfrac{p-1}{2}}}+o(\omega^{-\tfrac{p-1}{2}}).
\end{equation*} 
From the above asymptotics, sending $\omega$ to infinity, one gets that the limit of the expression in \eqref{delta'4} is positive and equals $\frac{p-5}{2}\int\limits_0^1(1-s^2)^{\tfrac{2}{p-1}}ds$. Hence   
the expression in \eqref{delta'4} is positive for $\omega$ large enough. This ensures the existence of $\omega_2$ such that $\partial_\lambda^2 E_\gamma(\varphi_\gamma^\lambda)|_{\lambda=1}\leq 0$ for $\omega\in [\omega_2,\infty).$

\textbf{2.}  Let now $\varphi_\gamma:=\varphi_{\omega,\gamma}^{odd}$. We need to show that $\partial_\lambda^2 E_\gamma(\varphi_\gamma^\lambda)|_{\lambda=1}\leq 0$ for $\omega\in[\omega_3,\infty).$
 The proof repeats the one of Lemma \ref{second_der}. The only difference is that  inequality \eqref{sec_1} has to be substituted by $$\frac{1}{\gamma}|\varphi_\gamma(0+)-\varphi_\gamma(0-)|^2<\frac{(p-5)(p-1)}{2(p+1)}\|\varphi_\gamma\|_{p+1}^{p+1}$$ and $\xi(\omega,\gamma)=\frac{2}{\gamma\sqrt{\omega}}.$

\section{Appendix}\label{sec6}
Let $\mc I\subseteq \mathbb{R}$ be an open interval. We say that  the function   $g(s): \mc I\to L^2(\Gamma)$ is \textit{$L^2$-differentiable  on $\mc I$} if the limit $\frac{d}{ds}g(s):=\lim\limits_{h\to 0}\frac{g(s+h)-g(s)}{h}$ exists in $L^2(\Gamma)$  for any $s\in \mc I$.
Below we give a sketch of the proof of the following "product rule".
\begin{proposition}\label{leibniz}
Let operator $H$ be defined by \eqref{D_alpha},  and the function  $g(s): \mc I\to L^2(\Gamma)$ be $L^2$-differentiable on the open interval $\mc I$, then we have
\begin{equation}\label{well_34}
\frac{d}{ds}\left[-i(H+m)^{-1}e^{iHs}g(s)\right]=e^{iHs}g(s)-m(H+m)^{-1}e^{iHs}g(s)-i(H+m)^{-1}e^{iHs}\frac{d}{ds}g(s).
\end{equation}
\end{proposition}  
\begin{proof}
Denote $F(s)=-i(H+m)^{-1}e^{iHs}g(s)$, then $\frac{d}{ds}F(s)=\lim\limits_{h \to 0}\frac{F(s+h)-F(s)}{h}.$
We have
\begin{equation}\label{well_35}
\begin{split}
&\frac{F(s+h)-F(s)}{h}\\&=\frac{1}{h}\left\{-i(H+m)^{-1}e^{i(H+m)(s+h)}e^{-im(s+h)}g(s+h)+i(H+m)^{-1}e^{i(H+m)s}e^{-ims}g(s)\right\}\\&= -i(H+m)^{-1}\frac 1{h}\left\{e^{i(H+m)(s+h)}-e^{i(H+m)s}\right\}e^{-im(s+h)}g(s+h)\\&-i(H+m)^{-1}e^{i(H+m)s}\frac 1{h}\left\{(e^{-im(s+h)}-e^{-ims})g(s+h)+e^{-ims}(g(s+h)-g(s))\right\}.
\end{split}
\end{equation}
To prove the assertion we need to analyze three last terms of \eqref{well_35}, that is we are aimed to prove that 
\begin{equation*}
\begin{split}
&-i(H+m)^{-1}\frac 1{h}\left\{e^{i(H+m)(s+h)}-e^{i(H+m)s}\right\}e^{-im(s+h)}g(s+h)\,\,\,\,\longrightarrow\,\,\,\, e^{iHs}g(s),\\
&-i(H+m)^{-1}e^{i(H+m)s}\frac1{h}(e^{-im(s+h)}-e^{-ims})g(s+h)\,\,\,\,\longrightarrow\,\,\,\, -m(H+m)^{-1}e^{iHs}g(s),\\
&-i(H+m)^{-1}e^{i(H+m)s}\frac1{h}e^{-ims}(g(s+h)-g(s))\,\,\,\,\longrightarrow\,\,\,\, -i(H+m)^{-1}e^{iHs} \frac{d}{ds}g(s)
\end{split}
\end{equation*}
in $L ^2(\Gamma)$ as $h\to 0$.

$\bullet$  By the Spectral Theorem for the self-adjoint operator $H$ we have:
\begin{equation}\label{well_36}
\begin{split}
&\|-i(H+m)^{-1}\frac 1{h}(e^{i(H+m)(s+h)}-e^{i(H+m)s})e^{-im(s+h)}g(s+h)-e^{iHs}g(s)\|_2^2\\&\leq
2\|-i(H+m)^{-1}\frac 1{h}(e^{i(H+m)(s+h)}-e^{i(H+m)s})e^{-im(s+h)}g(s+h)-e^{iHs}g(s+h)\|_2^2\\&+ 2\|g(s+h)-g(s)\|_2^2\\&\leq 2\int\limits_{\mathbb{R}}|-i(z+m)^{-1}\frac 1{h}(e^{i(z+m)(s+h)}-e^{i(z+m)s})e^{-im(s+h)}-e^{izs}|^2d(E_H(z)g(s+h), g(s+h))\\&+2\|g(s+h)-g(s)\|_2^2,
\end{split}
\end{equation}
where $E_H(z)$ is the spectral measure associated with $H$.
Denote by $f_h(z)$ the function under the integral in the above inequality. Making trivial manipulations one may show that $f_h(z)=e^{izs}(e^{iz\overline{h}}e^{im(\overline{h}-h)}-1)$, where  $\overline{h}$ lies between $0$ and $h$. It is obvious that $f_h(z)$ is bounded and converges to zero pointwise as $h\to 0$. Observing that
 $$(E_H(M)g(s+h),g(s+h))\underset{h\to 0}{\longrightarrow}(E_H(M)g(s),g(s))$$ for any  Borel set $M$, and using the  Dominated Convergence Theorem, we conclude that expression \eqref{well_36} tends to zero as $h\to 0$.
Finally, $-i(H+m)^{-1}\frac1{h}(e^{i(H+m)(s+h)}-e^{i(H+m)s)})e^{-im(s+h)}g(s+h)$ tends to $e^{iHs}g(s)$.

$\bullet$ Using boundedness of the resolvent $(H+m)^{-1}$ we get 
\begin{equation*}
\begin{split}
&\|-i(H+m)^{-1}e^{i(H+m)s}\frac 1{h}\left\{e^{-im(s+h)}-e^{-ims}\right\}g(s+h)+m(H+m)^{-1}e^{i(H+m)s}e^{-ims}g(s)\|_2\\&\leq
\|(H+m)^{-1}e^{i(H+m)s}\left\{-i\frac 1{h}(e^{-im(s+h)}-e^{-ims})g(s+h)+me^{-ims}g(s)\right\}\|_2\\&\leq C\left|-i\frac 1{h}(e^{-im(s+h)}-e^{-ims})+me^{-ims}\right\||g(s+h)\|_2+C\left|me^{-ims}\right\||g(s+h)-g(s)\|_2.
\end{split}
\end{equation*}
The expression above obviously tends to zero and therefore $-i(H+m)^{-1}e^{i(H+m)s}\frac 1{h}(e^{-im(s+h)}-e^{-ims})g(s+h)$ tends to $-m(H+m)^{-1}e^{iHs}g(s)$.

$\bullet$ Finally, estimating the last term in \eqref{well_35}
\begin{equation*}
\begin{split}
&\|-i(H+m)^{-1}e^{i(H+m)s}e^{-ims}\left\{\frac 1{h}(g(s+h)-g(s))-\frac{d}{ds}g(s)\right\}\|_2\\&\leq C\|\frac 1{h}(g(s+h)-g(s))-\frac{d}{ds}g(s)\|_2,
\end{split}
\end{equation*}
we get that $-i(H+m)^{-1}e^{i(H+m)s}e^{-ims}\frac 1{h}(g(s+h)-g(s))$ tends to $-i(H+m)^{-1}e^{iHs}\frac{d}{ds}g(s)$.
Summarizing the estimates of the three last terms in \eqref{well_35}, we finally obtain formula \eqref{well_34}.

\end{proof} 
\section*{Acknowledgements.}
This work started when M.O. visited IME-USP with the support of FAPESP (project: 17/17698-1). 
M.O. would like to thank Jaime Angulo Pava for his warm hospitality.


\begin{thebibliography}{BN} 
\bibitem{AdaNoj13}
R. Adami\ and\ D. Noja,  {\it Stability and symmetry-breaking bifurcation for the ground states of a NLS with a $\delta'$ interaction}, Comm. Math. Phys. {\bf 318} (2013), no.~1, 247--289.

%\bibitem{AdaNoj15}
% R. Adami, C. Cacciapuoti, D. Finco, D. Noja, {\it Stable standing waves for a NLS on star graphs as local %minimizers of the constrained energy}, J. Differential Equations {\bf 260} (2016), no.~10, 7397--7415.

\bibitem{AdaNoj14}
R. Adami, C. Cacciapuoti, D. Finco, D. Noja,\textit{ 
Variational properties and orbital stability of standing waves for NLS equation on a star graph}, J. Differential Equations {\bf 257} (2014), no.~10, 3738--3777. 

\bibitem{AdaNoj16}
R. Adami, C. Cacciapuoti, D. Finco, and D. Noja, \textit{Stable standing waves
for a NLS on star graphs as local minimizers of the constrained energy}, J.
Differential Equations {\bf 260} (2016), no.~10, 7397 -- 7415.

\bibitem{AlbBrz95}
S. Albeverio, Z. Brzezniak, L. Dabrowski,\textit{ Fundamental Solution of the Heat and Schr\"odinger Equations
with Point Interaction}, J. Funct. Anal. \textbf{130}  (1995), 220--254.

\bibitem{AlbGes05}
 S. Albeverio, F. Gesztesy, R. Hoegh-Krohn and H. Holden, {\it Solvable models in quantum mechanics}, second edition, AMS Chelsea Publishing, Providence, RI, 2005. 


\bibitem{AngGol18}
J. Angulo, N. Goloshchapova, \textit{Extension theory approach in the stability of the standing waves for the NLS equation with point interactions on a star graph}, Advances in Differential Equations {\bf 23}  (2018), 793--846. 




%

\bibitem{AngGol18a} J. Angulo,  N. Goloshchapova, \textit{On the orbital instability of excited states for the NLS equation with the $\delta$-interaction on a star graph}, 	Discrete  $\&$ Continuous Dynamical Systems {\bf 38} (2018), no.~ 10, 5039--5066.
 



 \bibitem{BanIgn14}
V. Banica, L.I.Ignat,\textit{ Dispersion for the Schr\"odinger equation on the line with multiple Dirac delta potentials and on delta trees}, Anal. Partial Differ. Equ. {\bf 7} (2014), no.~4,  903--927.

%\bibitem{BK}
%G. Berkolaiko\ and\ P. Kuchment, {\it Introduction to quantum graphs}, Mathematical Surveys and Monographs, 186, %Amer. Math. Soc., Providence, RI, 2013.

\bibitem{BerCaz81} H. Berestycki and T. Cazenave, \textit{Instabilit\'{e} des \'{e}tats stationnaires dans
les \'{e}quations de Schr\"{o}dinger et de Klein-Gordon non lin\'{e}aires}, C. R.
Acad. Sci. Paris S\'{e}r. I Math. \textbf{293} (1981), 489--492.



\bibitem{CacFin17} 
C. Cacciapuoti, D. Finco and D. Noja, \textit{Ground state and orbital stability for the NLS equation on a general starlike graph with potentials}, Nonlinearity {\bf 30} (2017), 3271--3303.



\bibitem{Caz03} T. Cazenave, \textit{Semilinear Schr\"odinger equations}, Courant Lect. Notes in Math.,
10, New York University, Courant Institute of Mathematical Sciences, New York;
Amer. Math. Soc., Providence, RI, 2003.
 

\bibitem{FukayaO} N. Fukaya and M. Ohta, 
\textit{Strong instablity of standing waves for nonlinear Schr\"odinger equations 
with attractive inverse power potential}, 
preprint, arXiv:1804.02127. 


\bibitem{FukJea08} R. Fukuizumi, L. Jeanjean,  \textit{Stability of standing waves for a nonlinear Schr\"odinger equation with a repulsive Dirac delta potential}, Discrete Contin. Dyn. Syst. {\bf 21} (2008), no.~1, 121--136. 

 \bibitem{FukOht08}  R. Fukuizumi, M. Ohta and T. Ozawa,  \textit{Nonlinear Schr\"odinger equation with a point defect}, Ann. Inst. H. Poincar\'e Anal. Non Lin\'eaire {\bf 25} (2008), no.~5, 837--845. 
 
\bibitem{Kai17}
 A. Kairzhan,
\textit{Orbital instability of standing waves for NLS equation on star graphs}, Proc. Amer. Math. Soc. \textbf{147} (2019), 2911--2924. 
 
 
 \bibitem{Coz08}
 S. Le Coz, \textit{A note on Berestycki-Cazenave's classical instability result
for nonlinear Schr\"odinger equations}, Adv. Nonlinear Stud. \textbf{8} (2008),
455--463.

\iffalse

\bibitem{GrilSha87} M. Grillakis, J. Shatah and W. Strauss, \textit{Stability theory of solitary waves in the presence of symmetry.} I, J. Funct. Anal. {\bf 74} (1987), no.~1, 160--197. 

\bibitem{GrilSha90}
M. Grillakis, J. Shatah and  W. Strauss, \textit{Stability theory of solitary waves in the presence of symmetry}. II, J. Funct. Anal. {\bf 94} (1990), no.~2, 308--348. 


\fi
 

\bibitem{CozFuk08}
S. Le Coz, R. Fukuizumi, G. Fibich, B. Ksherim and Y. Sivan, \textit{Instability
of bound states of a nonlinear Schr\"odinger equation with a Dirac potential},
Physica D {\bf 237} (2008), 1103--1128.


%\bibitem{Noj14} D. Noja, \textit{Nonlinear Schr\"odinger equation on graphs: recent results and open problems}, %Philos.
%Trans. R. Soc. Lond. Ser. A Math. Phys. Eng. Sci. {\bf 372} (2014), 20130002.



\bibitem{Oht18a} 
M. Ohta, 
\textit{Strong instability of standing waves for nonlinear Schr\"odinger equations with harmonic potential}, Funkcialaj Ekvacioj \textbf{61} (2018), 135--143.

\bibitem{Oht18} M. Ohta, 
{ \it Instability of standing waves for nonlinear Schr\"odinger equations with delta potential},  Sao Paulo Journal of Mathematical Sciences {\bf 13} (2019), 465--474.


\bibitem{OhtYam15}
M. Ohta and T. Yamaguchi, \textit{Strong instability of standing waves for nonlinear Schr\"odinger equations with double power nonlinearity}, SUT J. Math. \textbf{51} (2015), 49--58.


\bibitem{OhtYam16} M. Ohta and T. Yamaguchi, {\it Strong instability of standing waves for nonlinear Schr\"odinger equations with a delta potential}, RIMS Kokyuroku Bessatsu \textbf{B56} (2016), 79--92.




\end{thebibliography}
\end{document}